\documentclass[12pt]{amsart}
\usepackage{amsthm,amsmath,a4,amssymb,verbatim,enumerate}

\newtheorem{thm}{Theorem}[section]
\newtheorem{cor}[thm]{Corollary}
\newtheorem{lem}[thm]{Lemma}
\newtheorem{prop}[thm]{Proposition}
\theoremstyle{definition}
\newtheorem{defn}[thm]{Definition}
\theoremstyle{remark}
\newtheorem{rem}[thm]{Remark}
\newtheorem{ex}[thm]{Example}
\numberwithin{equation}{section}

\newcommand{\Hom}{\textnormal{Hom}}

\newcommand{\Kill}{\textnormal{Kill}}

\newcommand{\Ker}{\textnormal{Ker}}

\newcommand{\eps}{\varepsilon}

\newcommand{\g}{\mathfrak{g}}

\newcommand{\mk}{\mathfrak}
\newcommand{\bta}{\bar{\eta}}
\newcommand{\Sym}{\mathrm{Sym}}
\newcommand{\Alt}{\mathrm{Alt}}
\newcommand{\HH}{\mathrm{HH}}

\newcommand{\HC}{\mathrm{HC}}
\newcommand{\coker}{\mathrm{coker}}

\newcommand{\ot}{\otimes}
\newcommand{\we}{\wedge}
\newcommand{\cc}{\circledcirc}

\newcommand{\mkd}{\mathfrak{d}}
\newcommand{\mka}{\mathfrak{a}}
\newcommand{\mkv}{\mathfrak{v}}
\newcommand{\mkh}{\mathfrak{h}}

\newcommand{\mkl}{\mathfrak{l}}
\newcommand{\mks}{\mathfrak{s}}

\newcommand{\bd}{\partial}
\newcommand{\la}{\langle}
\newcommand{\ra}{\rangle}

\begin{document}

\address{CNRS -- D\'epartement de Math\'ematiques, Universit\'e Paris-Sud, 91405 Orsay, France}

\email{yves.cornulier@math.u-psud.fr}

\title{On the Koszul map of Lie algebras}
\author{Yves Cornulier}%
\date{May 21, 2014}





\begin{abstract}
We motivate and study the reduced Koszul map, relating the invariant bilinear maps on a Lie algebra and the third homology. We show that it is concentrated in degree 0 for any grading in a torsion-free abelian group, and in particular it vanishes whenever the Lie algebra admits a positive grading. We also provide an example of a 12-dimensional nilpotent Lie algebra whose reduced Koszul map does not vanish. In an appendix, we reinterpret the results of Neeb and Wagemann about the second homology of current Lie algebras, which are closely related to the reduced Koszul map.
\end{abstract}
\maketitle

\section{Introduction}

Let $\g$ be a Lie algebra. Although we are mostly interested in the case when $\g$ is a Lie algebra over a field of characteristic zero, we assume by default that $\g$ is a Lie algebra over an arbitrary commutative ring $R$ (all commutative rings in the paper are assumed to be associative with unit).

All tensors below are over the ground ring $R$ (unless explicitly mentioned). Recall that the homology of $\g$ is defined as the homology of the Chevalley-Eilenberg complex $\Lambda^*\g$; in particular $H_i(\g)$ is a subquotient of $\Lambda^i\g$.

We consider the symmetric square $S^2\g=\g\cc\g$. We define the Killing module $\Kill(\g)$ as the cokernel of the map from $\g^{\ot 3}$ to $S^2\g$ mapping $x\ot y\ot z$ to $[x,y]\cc z-x\cc [y,z]$. In other words, this is the module of co-invariants of the $\g$-module $S^2\g$. In particular, for every $R$-module $M$, $\Hom_{R\textnormal{-mod}}(\Kill(\g),M)$ is canonically isomorphic to the $R$-module of invariant symmetric $R$-bilinear forms $\g\times\g\to M$; for instance, when $R$ is a field, the dual $\Kill(\g)^*=(\Sym^2\g)^\g$ is canonically isomorphic to the space of invariant symmetric bilinear forms on $\g$.

Let us consider the linear map $\check{\eta}:\g^{\ot 3}\to S^2\g$, mapping $x\ot y\ot z$ to $x\cc [y,z]$; it is alternating in the last 2 variables. If we consider the composite map $\g^{\ot 3}\to\Kill(\g)$, it is by definition invariant under cyclic permutations, and therefore factors to an $R$-module homomorphism $\eta:\Lambda^3\g\to\Kill(\g)$, called the {\bf Koszul map} (or homomorphism); it is also known as Cartan-Koszul map. A straightforward verification (see Lemma \ref{vanb3}) shows that $\eta$ vanishes on 3-cycles.
Furthermore, we can restrict $\eta$ to the submodule of 3-cycles and thus obtain a factor map \[\bta:H_3(\g)\to\Kill(\g),\]
called the {\bf reduced Koszul map}.

Note that over a field and for finite-dimensional Lie algebras, the adjoint of the Koszul map is the map $J:(\Sym^2\g)^\g\to\Alt^3\g$ mapping an invariant symmetric bilinear form $B$ on $\g$ to the alternating trilinear form $J_B(x,y,z)=B([x,y],z)$; the adjoint of the reduced Koszul map is the resulting map $(\Sym^2\g)^\g\to H^3(\g)$.

\medskip

Let us mention a few important situations in which the Koszul map arises. (This is given as a motivation; the reader can skip this part and go to the statement of results on page~\pageref{pag}.)

\smallskip

\noindent{\bf $\bullet$ Semisimple Lie algebras.} Koszul \cite[\S 11]{Kos} considered the above map $J$ for an arbitrary Lie algebra $\g$ over a field of characteristic $\neq 2$ as a map $(\Sym^2\g)^\g\to(\Alt^3\g)^\g$, showing that it is injective if and only if $H^1(\g)=\{0\}$ and bijective if and only if $H^1(\g)=H^2(\g)=\{0\}$ (the injectivity assertion was already used by Chevalley-Eilenberg in the proof of \cite[Theorem 21.1]{ChE}). He defines $\g$ to be reductive (let us write ``{\em Koszul-reductive}") if it is finite-dimensional and $\Alt^*\g$ is absolutely completely reducible as a $\g$-module (this is the usual notion when the field has characteristic zero, but also includes some positive characteristic cases: for instance any Chevalley absolutely simple Lie algebra is Koszul-reductive in large enough characteristic). A straightforward fact is that, for arbitrary Lie algebras, invariant cochains are cocycles, i.e.\ $(\Alt^*\g)^\g\subset Z^*(\g)$; thus there is a canonical graded homomorphism $(\Alt^*\g)^\g\to H^*(\g)$, which was proved to be an isomorphism when $\g$ is Koszul-reductive (by Chevalley-Eilenberg \cite[Theorem 19.1]{ChE} in characteristic zero and Koszul in general, by the same method \cite[Lemme~9.1]{Kos}). On the other hand, for an arbitrary Lie algebra, if $H^1(\g)=\{0\}$, it is immediate that every invariant 2-cocycle is zero (see \cite[Theorem 21.1]{ChE}); thus if $\g$ is Koszul-reductive and perfect (that is, $H^1(\g)=\{0\}$, or equivalently since $\g$ is Koszul-reductive, $\g$ has trivial center) it also satisfies $H^2(\g)=\{0\}$. Combining these results, it follows that if $\g$ is perfect and Koszul-reductive, then the dual of the reduced Koszul map $J:(\Sym^2\g)^\g\to H^3(\g)$ is an isomorphism. This applies in particular when $K$ has characteristic zero and $\g$ is semisimple, yielding a description of the third (co)homology group of $\g$: indeed, the dimension of $(\Sym^2\g)^\g$ is then easily computable: it is the number of simple factors of $\g\ot\hat{K}$, where $\hat{K}$ is an algebraic closure of $K$. 

\smallskip

\noindent{\bf $\bullet$ The Pirashvili long exact sequences.}

Let $\g$ be a Lie algebra over a field of characteristic $\neq 2$. In \cite{Pir}, Pirashvili provides two (closely related) exact sequences:
\[\dots \to H_1^{\textnormal{rel}}(\g)\to HL_3(\g)\to H_3(\g)\stackrel{\bar{\eta}}\to H_0^{\textnormal{rel}}(\g)\to HL_2(\g)\to H_2(\g)\to 0;\]
\[\dots \to HR_1(\g)\to H_2(\g,\g)\stackrel{\bar{p_3}}\to H_3(\g)\stackrel{\bar{\eta}}\to HR_0(\g)\stackrel{\bar{s}}\to H_1(\g,\g)\stackrel{\bar{p_2}}\to H_2(\g)\to 0.\]
where both $H_0^{\textnormal{rel}}(\g)$ and $HR_0(\g)$ are canonically isomorphic to $\Kill(\g)$. Let us partially indicate the meaning of the terms: $HL_*(\g)$ is the Leibniz homology of $\g$ (discovered by Loday), which is the homology of the Leibniz complex, a complex lifting the usual Chevalley-Eilenberg complex but using tensor powers instead of the exterior algebra \cite[\S 10.6]{Lod}. The map $\bar{s}$ is induced by the natural map $s:S^2\g\to\g^{\ot 2}$ defined by $s(x\cc y)=x\ot y+y\ot x$, and the map $\bar{p_i}$ is induced by the natural projection $\g^{\ot i}\to\Lambda^i\g$.

The second sequence, in particular, shows that $\bar{\eta}=0$ if and only if $\bar{p_3}$ is surjective, if and only if $\bar{s}$ is injective, in which case we get an exact sequence
\[0\to \Kill(\g)\to H_1(\g,\g)\to H_2(\g)\to 0.\]
At the opposite, $\bar{\eta}$ is surjective if and only if $\bar{s}=0$, in which case $\bar{p_2}$ is an isomorphism from $H_1(\g,\g)$ onto $H_2(\g)$.

In practice, since $\bar{\eta}$ is often convenient to handle, a good understanding of it (such as vanishing results as we provide below) provides information on the other maps involved, e.g.\ the maps $\bar{s}:\Kill(\g)\to H_1(\g,\g)$ or the projection $H_2(\g,\g)\to H_3(\g)$; this is used, for instance, in \cite{HPL}.

\smallskip

\noindent{\bf $\bullet$ 2-homology of current Lie algebras}

Let $\g$ be a Lie algebra over a field $K$ of characteristic $\neq 2$, and $A$ a commutative, associative, unital $K$-algebra. The Lie algebra $A\ot\g$ (where $\ot$ means $\ot_K$), viewed as Lie algebra over $K$, is called current Lie algebra. For every $i$ we have a natural $K$-linear map $H_i(A\ot \g)\to A\ot H_i(\g)$; for $i\le 1$ it is an isomorphism and for $i\le 2$ it is surjective; still, for $i=2$ it is not necessarily injective, and the determination of $H_2(A\ot\g)$ is interesting. Let us assume, to simplify, that $H_1(\g)=H_2(\g)=\{0\}$. 
Let us write $\bar{\eta}_\g$ to emphasize $\g$. The results of Neeb and Wagemann can then be used to yield the following exact sequence:
\[0 \to \HC_1(A)\ot\mathrm{Im}(\bar{\eta}_\g)\to H_2(A\ot\g) \to \HH_1(A)\ot\coker(\bar{\eta}_\g)\to 0.\]
(See Appendix \ref{nwd}, where convenient interpretations of the Neeb-Wagemann results are provided.)

Here $\HH_1(A)$ is the first Hochschild homology group and $\HC_1$ the first cyclic homology group. Instead of defining them here, we mention two important fundamental examples of $A$, assuming that $K$ has characteristic zero, namely the polynomial ring and the ring of Laurent polynomials:
\[\HC_1(K[t])\simeq\{0\},\;\HC_1(K[t^{\pm 1}])\simeq K,\quad \HH_1(K[t])\simeq K[t],\; \HH_1(K[t^{\pm 1}])\simeq K[t^{\pm 1}].\]
%
%
%

For instance, when $A=K[t]$ (resp.\ $K[t^{\pm 1}]$), $H_2(A\ot\g)$ is zero (resp.\ finite-dimensional) if and only if the reduced Koszul map $\bar{\eta}$ is surjective, and infinite-dimensional otherwise.

The above exact sequence degenerates in two extreme cases: the first is when the reduced Koszul map $\bar{\eta}$ is surjective: then we get
\[H_2(A\ot\g)\simeq \HC_1(A)\ot\Kill(\g).\]
This is actually the best-known case: indeed it holds for semisimple Lie algebras when $K$ has characteristic zero, in which case the above isomorphism is notably due to Bloch \cite{Blo} and Kassel-Loday \cite{KasL}. 

The other extreme case is when $\bar{\eta}=0$: then we get 
\[H_2(A\ot\g)\simeq \HH_1(A)\ot\Kill(\g).\]

It turns out that the vanishing condition of the reduced Koszul map is very common and the semisimple case is quite peculiar. For finite-dimensional $\g$ with $K$ of characteristic zero, the conditions of vanishing of $H_1(\g)$ and of $\bar{\eta}$, imply that $\g$ is both perfect and solvable and thus only hold for $\g=\{0\}$. Still, these conditions can be made compatible because all the given exact sequences hold when $\g$ is endowed with a grading (in an arbitrary Lie algebra). Then we can fix a weight $\alpha$, so that under the assumption of vanishing of $H_1(\g)_\alpha$ and $H_2(\g)_\alpha$, the previous exact sequence holds in degree $\alpha$. This can be applied in many cases including finite-dimensional examples, for instance the example of Appendix \ref{appcheck}.

\begin{rem}
The Neeb-Wagemann computation of the second (co)homology of current Lie algebras contradicts a description provided 14 years earlier by Zusmanovich \cite{Zus} but this was not noticed until now. In Appendix \ref{appcheck}, we make an explicit computation in a particular case $\mk{l}$ (some perfect but non-semisimple 6-dimensional complex Lie algebra) and $A=\mathbf{C}[t]$, in which $H_2(\mathbf{C}[t]\otimes\mk{l})$ does not vanish, in accordance with the results in \cite{NW} but confirms that the argument given in \cite{Zus}, extending the previously known semisimple case, is mistaken.
\end{rem}

\bigskip

\label{pag}In this paper, we are interested in the reduced Koszul map for more general Lie algebras, especially nilpotent Lie algebras. For instance it is obviously zero for abelian Lie algebras $\mka$ (although $H_3(\mka)$ and $\Kill(\mka)$ are nonzero in general). Actually it remains zero in a considerably greater generality. Magnin \cite{Mag} investigated this phenomenon for some classes of Lie algebras over fields, calling a Lie algebra {\em I-null} if the Koszul map is zero (this just means that all invariant symmetric bilinear forms on $\g$ factor through $S^2(\g/[\g,\g])$, and {\em I-exact} if the reduced Koszul map is zero (obviously I-null implies I-exact); actually in Magnin's setting, $I$ denotes the above map $B\mapsto J_B$ and to say that the reduced Koszul map is zero means that $J_B$ in an exact 3-cocycle for every $B\in\Sym(\g)^\g$. Magnin checked, for instance, that all complex nilpotent Lie algebras up to dimension 7 have a vanishing reduced Koszul map, and also exhibits interesting classes of I-null Lie nilpotent algebras (for instance, the Lie algebra of $n\times n$ strictly upper-triangular matrices for any $n\ge 0$).

If $\g$ is a Lie algebra graded in an abelian group $A$, the tensor powers $\g^{\ot k}$ are naturally graded in $A$, namely by
\[(\g^{\ot k})_{\alpha}=\sum_{\beta_1+\dots+\beta_k=\alpha}\g_{\beta_1}\otimes \dots\otimes\g_{\beta_k},\]
this grading is compatible with taking exterior or symmetric powers, and is preserved by the boundary maps and the Koszul map, and thus the reduced Koszul map is naturally a graded $R$-module homomorphism. 

\begin{thm}\label{vanish}
Assume that the ground ring $R$ is a commutative $\mathbf{Q}$-algebra (e.g.\ a field of characteristic zero).
Let $\g$ be a Lie algebra graded in a torsion-free abelian group $A$. Then the reduced Koszul map is zero in every nonzero degree $\alpha$.
\end{thm}

(We refer to Section \ref{vagra} for a statement applying to fields of positive characteristic.)

Although here we emphasize the study of the vanishing of the reduced Koszul map, another point of view consists in considering Lie algebras endowed with an invariant symmetric bilinear form $B$: 

\begin{cor}\label{cor3form}
Let $\g$ be a Lie algebra over a field of characteristic zero, with an invariant symmetric bilinear form $B$. Suppose that $\g$ admits a grading in a torsion-free abelian group for which $B$ is homogeneous of nonzero degree. Then the 3-form $\eta^*(B)$ is exact.
\end{cor}

An elementary particular case (already observed by Neeb and Wagemann) is the case of the coadjoint semidirect product $\g\ltimes\g^*$, where $\g$ is an arbitrary Lie algebra over a field (of characteristic $\neq 2,3$), and $\g^*$ its coadjoint representation (viewed as an abelian ideal): indeed the canonical bilinear form has degree -1 for the natural grading in which $\g$ has degree 0 and the ideal $\g^*$ has degree~1. 

Let us provide some particular cases of Theorem \ref{vanish}.

\begin{cor}\label{gradnn}
Assume that $R$ is a commutative $\mathbf{Q}$-algebra. If $\g$ is graded in the set of non-negative integers $\mathbf{N}$, then the reduced Koszul map of $\g$ coincides with the reduced Koszul map of $\g_0$.
\end{cor}

Note that this is not true when $\g$ is graded in $\mathbf{Z}$: indeed, the simple Lie algebra $\mk{sl}_2$ has a grading in $\{-1,0,1\}$, has a nonzero reduced Koszul map, but $\g_0$ is 1-dimensional abelian so obviously has a zero reduced Koszul map.

\begin{cor}\label{nstar}
Under the same assumption on $R$, if $\g$ is graded in the set of positive integers $\mathbf{N}^*$, then the reduced Koszul map of $\g$ is zero.
\end{cor}

Note that when the set of weights (i.e., those $\alpha$ such that $\g_\alpha\neq\{0\}$) is finite (e.g., $\g$ is finite-dimensional over a field), then having a grading on $\mathbf{N}^*$ implies that $\g$ is nilpotent. For instance, Corollary \ref{nstar} provides:

\begin{cor}\label{cor2nilp}
Suppose that the ground ring is a field of characteristic zero. If $\g$ is a 2-nilpotent Lie algebra (that is $[\g,[\g,\g]]=\{0\}$), then the reduced Koszul map of $\g$ is zero.
\end{cor}

Indeed, such a Lie algebra admits a grading in $\{1,2\}$. Nevertheless, the statement of Corollary \ref{cor2nilp} is shown to hold over an arbitrary field of characteristic $\neq 3$ (as a consequence of Theorem \ref{reki4}). Counterexamples when the ground ring is a field of characteristic 3 are provided in Proposition \ref{char3nonz}; counterexamples when the ground ring is a $\mathbf{Q}$-algebra but has nonzero nilpotent elements are provided in Proposition \ref{nonredu}.

Actually, the class of nilpotent Lie algebras with a grading on $\mathbf{N}^*$ is considerably larger. It includes, in particular, the {\em Carnot-graded Lie algebras}, namely those Lie algebras over a field, graded in $\mathbf{N}^*$ such that $[\g_1,\g_i]=\g_{i+1}$ for all $i\ge 1$. A Lie algebra admitting such a grading is called {\em Carnot}. For instance, if $R$ is a field, then every 2-nilpotent Lie algebra over $R$ is Carnot.

In turn, the class of Lie algebras admitting a grading in $\mathbf{N}^*$ is much larger than the class of Carnot Lie algebras. For instance, over a field of characteristic zero, every nilpotent Lie algebra of dimension $\le 6$ admits a grading in $\mathbf{N}^*$, while there are exactly two non-isomorphic five-dimensional non-Carnot nilpotent complex Lie algebras ($L_{5,5}$ and $L_{5,6}$ in \cite{graaf}).

Let us also provide one application in a case which is not graded, but closely related: let $\mk{v}$ be an $R$-module graded in $\mathbf{Z}$. Define its right completion to be the $R$-module $\overline{\mk{v}}$ of sequences $\sum_{n\in\mathbf{Z}}v_n$ with $v_n\in\mk{v}_n$ and $v_{-n}=0$ for $n$ large enough. If $\mkv=\g$ is a Lie algebra and the grading is a Lie algebra grading, then $\overline{\g}$ inherits the structure of a Lie algebra in the natural way.

The following is not obtained as a corollary of Theorem \ref{vanish}, but by running the same proof.

\begin{cor}\label{rcomgr}
Suppose that the ground ring is a $\mathbf{Q}$-algebra and that $\g$ is graded in $\mathbf{Z}$, with right completion $\overline{\g}$. Then the reduced Koszul map of $\overline{\g}$ is concentrated in degree 0, in the sense that for every 3-cycle $c\in Z_3(\overline{\g})$ written as $\sum_{n\in\mathbf{Z}} c_n$ with $c_n\in Z_3(\g)_n$, we have $\eta(c)=\eta(c_0)$. In particular, if $\g$ is graded in $\mathbf{N}^*$, then the reduced Koszul map of $\overline{\g}$ is zero, and if $\g$ is graded in $\mathbf{N}$, then the reduced Koszul map of $\overline{\g}$ coincides with that of $\g_0$.
\end{cor}

Right completions of Lie algebras graded in $\mathbf{N}^*$ occur in many places in the literature, as particular but common instances of pro-nilpotent Lie algebras.

On the other hand, say over a field of characteristic zero, there exist finite-dimensional nilpotent Lie algebras with no grading in $\mathbf{N}^*$. For instance, {\em characteristically nilpotent} Lie algebras, namely Lie algebras in which every derivation is nilpotent, have no nontrivial grading in $\mathbf{Z}$. Indeed if $\g$ has a grading in $\mathbf{Z}$, then the operator of multiplication by $i$ on $\g_i$ is a derivation, and is nilpotent only if $\g=\g_0$. (Such Lie algebras exist in dimension $\ge 7$.)
However, case-by-case computations show that small-dimensional known characteristically nilpotent Lie algebras have a zero reduced Koszul map, and Fialowski, Magnin and Mandal \cite{FMM} asked whether there indeed exists a (complex) nilpotent Lie algebra with a zero reduced Koszul map. It took us some significant effort to find an example.

\begin{thm}
Let $K$ be any field. There exist a 12-dimensional 7-nilpotent Lie algebra over $K$ whose reduced Koszul map is nonzero.
\end{thm}

Actually, we check in Section \ref{small} that for nilpotent Lie algebras over a field of characteristic zero, up to dimension $9$ the reduced Koszul map is always zero (by reduction to regular quadratic Lie algebras and using in this case, using the classification in \cite{Kat}, the existence of a grading in positive integers and Corollary \ref{nstar}); the same seems to also hold in dimension 10 (we have not completely checked).

Concerning the nilpotency length, we do not know if the reduced Koszul map is zero for all complex 3-nilpotent Lie algebras; actually we check (Corollary \ref{meta3ni}) that this is equivalent to whether it vanishes for all complex metabelian Lie algebras (metabelian means 2-solvable, i.e.\ the derived subalgebra is abelian).

In Section \ref{solvable}, we address the solvable case:

\begin{thm}
Let $K$ be any field. There exists a 9-dimensional 3-solvable (central-by-metabelian) Lie algebra with nonzero reduced Koszul map; if $K$ has characteristic zero this is the smallest possible dimension for such an example.
\end{thm}

Theorem \ref{vanish} also has applications in the description of the second (co)homology of current Lie algebras addressed above, see Corollary \ref{iterat} (and the preceding remarks) and Corollary \ref{h10eta0}.

\medskip

\noindent {\bf Acknowledgements.} I am grateful to Louis Magnin for interesting discussions and for a computer checking of the construction of Section \ref{12dim}. I also thank Karl-Hermann Neeb and Friedrich Wagemann for useful discussions and comments; many thanks are due to the referee for pointing out many typos. 

\section{Preliminaries}

\subsection{Lie algebras, tensor products}

Let us work over an arbitrary commutative ring $R$.

Recall that Lie algebra is an $R$-module endowed with a bilinear map $[\cdot,\cdot]:\g\times\g\to\g$ satisfying $[x,x]=0$ for all $x$, and $\mathrm{jac}(x,y,z)=0$ for all $x$, where
\[\mathrm{jac}(x,y,z)=[x,[y,z]]+[y,[z,x]]+[z,[x,y]].\]
We denote by $\we$ the exterior product and by $\cc$ the symmetric product.

If $\g$ is a Lie algebra, we define its {\em lower central series} by
\[\g^{(1)}=\g,\quad \g^{(i+1)}=[\g,\g^{(i)}].\]
It is a Lie algebra filtration, in the sense that $[\g^{(i)},\g^{(j)}]\subset \g^{(i+j)}$ for all $i,j$, and moreover the above inclusion is always an equality. We also define the {\em derived series} by $D^0\g=\g$, $D^{i+1}\g=[D^i\g,D^i\g]$. By an obvious induction, we have $D^i\g\subset \g^{(2^i)}$ for all $i\ge 0$.

Recall that $\g$ is {\em $k$-nilpotent} if $\g^{(k+1)}=\{0\}$ and {\em nilpotent} if is $k$-nilpotent for some $k$; the smallest $k$ for which this holds is called nilpotency length (or nilindex) of $\g$: it is 0 by definition for the zero Lie algebra and 1 for nonzero abelian Lie algebras.

\subsection{Killing module}

If $\g$ is Lie algebra, we define its {\em Killing module} $\Kill(\g)$ the space of co-invariants of the $\g$-module $S^2\g$. Thus $\Kill(\g)=\coker(T)$, where 
\[T:\Lambda^2\g\ot\g\to S^2\g,\quad ((x\we y)\ot z)\mapsto x\cc [y,z]-y\cc [z,x]\]
We also write $x\equiv y$, for $x,y\in S^2\g$, to mean that $x$ and $y$ represent the same element of $\Kill(\g)$, i.e.\ $x-y\in\mathrm{Im}(T)$.

The terminology is more a tribute to Wilhelm Killing (1847-1923) than a direct reference to the Killing form: the latter only makes sense for finite-dimensional Lie algebras over field, and for instance is always zero in the nilpotent case, while $\Kill(\g)$ can be interestingly complicated; it is also denoted by $B(\g)$ by \cite{Had,Zus}, but we avoid this notation in order to avoid confusion, especially with the group of boundaries. Actually, $\Kill(\g)$ is a ``predual" to the more familiar concept of the space of symmetric bilinear forms: for instance when $R=K$ is a field, the dual $\Kill(\g)^*$ is then the space of invariant symmetric bilinear forms on $\g$. Any homomorphism between Lie algebras $f:\g\to\mkh$ functorially induces an $R$-module homomorphism $\Kill(f):\Kill(\g)\to\Kill(\mkh)$, which is surjective as soon as $f$ is surjective.

\subsection{Koszul map}

We consider the {\em raw Koszul map} as the map $\check{\eta}:\g\ot\Lambda^2\g\to S^2\g$ defined by $\check{\eta}(x\ot (y\we z))=x\cc [y,z]$, and the composite map into $\Kill(\g)$ clearly factors to a map $\eta:\Lambda^3\g\to\Kill(\g)$, called the {\em Koszul map}. 

The Koszul map is especially known at a dual level: assuming that $R$ is a field, the adjoint $\eta^*$ maps an invariant symmetric bilinear form $B$ on $\g$ to the trilinear alternating form $J_B=\eta^*(B)$ defined by $J_B(x,y,z)=B(x,[y,z])$.

A well-known observation is the following:

\begin{lem}\label{vanb3}
The Koszul map $\eta$ vanishes on $B_3(\g)$.
\end{lem}
\begin{proof}
Recall that
\[\partial_n(x_1\we\dots\we x_n)=\sum_{i<j}(-1)^{i+j}[x_i,x_j]\we x_1\we\dots\we \hat{x_i}\we\dots\we\hat{x_j}\we\dots \we x_n;\] define a map $\check{\partial_n}:\g^{\ot n}\to\g^{\ot n-1}$ by the same formula, replacing signs $\we$ by $\ot$. Then we obtain
\[\check{\eta}\check{\bd_4}(x\ot y\ot z\ot t)=-2([x,y]\cc [z,t]+[x,z]\cc [t,y]+[x,t]\cc [y,z])\qquad\]
\[\qquad= 2\Big(t\cc \mathrm{jac}(x,y,z) + T\big(([x,y]\we t)\ot z+([z,x]\we t)\ot y+([y,z]\we t)\ot x\big)\Big).\]
Thus $\eta\bd_4=0$.
\end{proof}

At a dual level, again assuming that the ground ring is a field, Lemma \ref{vanb3} says that the image of $\eta^*$ consists of 3-cocycles. More generally, the lemma is equivalent to the well-known statement that for every $R$-module $\mk{m}$ (viewed as a trivial $\g$-module), the dual map $\Hom(\Kill(\g),\mk{m})\to\Hom(\Lambda^3\g,\mk{m})$ has an image consisting of 3-cocycles.

\begin{defn}
The {\em reduced Koszul map} $\bar{\eta}$ is the map $H_3(\g)\to\Kill(\g)$ induced by the restriction of $\eta$ to $Z_3(\g)$.
\end{defn}

\begin{rem}
In this paper, we are often mainly interested in the image of $\bar{\eta}$. In particular, to determine this image, it is enough to consider the restriction of $\eta$ to $Z_3(\g)$ and we do not have to bother with 3-boundaries. Still, of course Lemma \ref{vanb3} can be useful when we have information about 3-homology (e.g.\ vanishing).
\end{rem}

\subsection{Regular quadratic Lie algebras}\label{rqla}

Here, the ground ring is a field $K$.

\begin{defn}
A {\em regular quadratic} Lie algebra is a Lie algebra over $K$ endowed with an invariant nondegenerate symmetric bilinear form.

A Lie algebra is {\em quadrable} if it can be endowed with a structure of regular quadratic Lie algebra.
\end{defn}	

As far as I know, these notions were first introduced by Tsou and Walker \cite{TW}.
If $\mathrm{char}(K)\neq 2$, a Lie algebra $\g$ is quadrable if and only if the adjoint and coadjoint representation of $\g$ are isomorphic as $K\g$-modules \cite[Theorem 1.4]{MR93}. (This can look at first sight as a triviality, but there is a little issue in proving this condition implies that $\g$ is quadrable: formally, it only implies the existence of a nondegenerate bilinear form, not necessarily symmetric. However, an easy argument shows that $[\g,\g]$ is contained in the kernel of every invariant alternating bilinear form and this helps to conclude.)

\begin{rem}Regular quadratic Lie algebras have many other names in the literature, including: quadratic Lie algebras, (symmetric) self-dual Lie algebras, metrized/metric Lie algebras, orthogonal Lie algebras. Also, they sometimes denote Lie algebras admitting (at least) a non-degenerate invariant quadratic form, or Lie algebras endowed with such a form, or ambiguously both (still, Tsou-Walker and Astrakhantsev initially used two distinct words, namely ``metrized/\penalty0metric" and ``metrizable"). Also, almost all these terminologies are used for other meanings (quadratic Lie algebra for quadratically presented Lie algebra; metric Lie algebra for a Lie algebra endowed with a scalar product which is not necessarily invariant, orthogonal Lie algebra for the usual $\mk{so}_n$). Using ``regular quadratic", we follow the terminology from Favre-Santharoubane \cite{FS}.
\end{rem}

We recall the notion of double extension. 
Let $\mkh$ be a regular quadratic Lie algebra and let $D$ be a skew-symmetric self-derivation of $\mkh$. The {\bf double extension} $\g$ of $\mkh$ by $D$ \cite{FS,MR85} is defined as follows: as a space, it is $Ke\oplus \mkh\oplus Kf$, where $e,f$ are symbols representing generators of the lines $Ke$ and $Kf$. The scalar product is defined so that it extends the original scalar product on $\mkh$, the vectors $e,f$ are isotropic and orthogonal to $\mkh$, and $\la e,f\ra=1$. The new Lie bracket $[\cdot,\cdot]$ is defined so that $f$ is central, and, denoting $[x,y]'$ the original bracket on $\mkh$ 
\[[e,x]=Dx,\quad [x,y]=[x,y]'+\la Dx,y\ra f,\quad \forall x,y\in\mkh.\]
Note that if $\mkh$ is solvable then so is $\g$, and if $\mkh$ is nilpotent, then $\g$ is nilpotent if and only if $D$ is a nilpotent endomorphism. See \S\ref{small} for basic examples.

Actually, any finite-dimensional solvable regular quadratic Lie algebra can be obtained from an orthogonal space (viewed as an abelian regular quadratic Lie algebra) using iterated double extensions; this provides a classification scheme in small dimensions.

\subsection{Koszul map and regular quadratic Lie algebras}

The following lemma is easy but useful in the study of the Koszul map.

\begin{lem}\label{somq}
Let $\g$ be a Lie algebra over $K$. Suppose that the reduced Koszul map of $\g$ does not vanish. Then some quotient of $\g$ is quadrable and also has a nonzero reduced Koszul map.
\end{lem}
\begin{proof}
Suppose that $c\in Z_3(\g)$ does not belong to the kernel of $\eta_\g$ (the function $\eta$ of $\g$). Consider a linear form $f$ on $\Kill(\g)$ not vanishing on $\eta(c)$: it defines an invariant symmetric bilinear form $B$ on $\g$. The kernel of the symmetric bilinear form $B$ is an ideal $\mk{i}$ of $\g$; let $\mkh=\g/\mk{i}$ be the quotient. Clearly $B$ factors through a symmetric bilinear form $B'$ on $\mkh$, and $(\mkh,B')$ is a regular quadratic Lie algebra.
Moreover, since $c$ is a 3-cycle, its image $c'$ in $\Lambda^3\mkh$ is a 3-cycle as well, and $\eta_{\mkh}(c)$ is the projection of $\eta_\g(c)$. If $f'$ is the linear form on $\Kill(\mkh)$ defined by $B'$, then $f$ is the composite of the projection $\Kill(\g)\to\Kill(\mkh)$ and $f'$. Thus we have $f'(\eta_\mkh(c'))=f(\eta(c))\neq 0$. Hence $c'\notin\Ker(\eta_\mkh)$.
\end{proof}

\begin{rem}
A regular quadratic Lie algebra is endowed with a canonical alternating 3-form, namely the image of the symmetric bilinear form by the adjoint Koszul map, called canonical 3-form. It is invariant and hence closed. Its image in cohomology is called canonical 3-cohomology class. Regular quadratic Lie algebras for which the canonical 3-form is exact (i.e.\ whose canonical 3-cohomology class vanishes) are sometimes called {\em exact}. For instance, the regular quadratic Lie algebra $(\mkh,B')$ constructed in the proof of Lemma \ref{somq} is non-exact.

Note that being exact is a property of regular quadratic Lie algebras; but this is not a property of the underlying Lie algebra. Indeed, assume for simplicity that $K$ has characteristic zero, and let $(\g,B)$ be a non-exact regular quadratic Lie algebra (e.g., a simple Lie algebra). Consider the semidirect product $\g\ltimes\g^*$ (as described after Corollary \ref{cor3form}), where the dual $\g^*$ is viewed as an abelian ideal with the coadjoint representation; let $B'$ be the bilinear form given by duality (for which both $\g$ and $\g^*$ are isotropic); also view $B$ as a bilinear form on $\g\ltimes\g^*$, with kernel $\g^*$. Then for all but finitely many $\lambda\in K$, the form $B'+\lambda B$ is non-degenerate; on the other hand, since $\eta^*B'$ is exact and $\eta^*B$ is non-exact (by a simple argument using that the projection $\g\ltimes\g^*\to\g$ is split), $B'+\lambda B$ is non-exact whenever $\lambda\neq 0$. Hence $\g\ltimes\g^*$ carries both structures of exact and non-exact regular quadratic Lie algebra.  
\end{rem}

\subsection{Filtration on the Killing module}\label{fiki}

Let again $R$ be arbitrary. There is a natural filtration on $\Kill(\g)$, defined as follows: for $i\ge 2$, we let $\Kill^{(i)}(\g)$ be the image of $\g\ot\g^{(i-1)}$ in $\Kill(\g)$. Thus 
\[\Kill(\g)=\Kill^{(2)}(\g)\supset \Kill^{(3)}(\g)\supset\dots\] 

The quotient $\Kill^{(2)}(\g)/\Kill^{(3)}(\g)$ is obviously isomorphic to the symmetric square $S^2(\g/[\g,\g])$. Thus the most interesting part of $\Kill(\g)$ lies in $\Kill^{(3)}(\g)$, which is the image of the Koszul map $\eta$. Again, $\Kill^{(i)}$ can be viewed as a functor, mapping surjections to surjections.

If $\g$ is $k$-nilpotent then $\Kill^{(k+2)}(\g)=\{0\}$. More generally, we have:

\begin{lem}\label{killni}
For every $i\ge 0$, the projection $\Kill(\g)\to\Kill(\g/\g^{(i+1)})$ induces an isomorphism of $R$-modules
\[\qquad\qquad\Kill(\g)/\Kill^{(i+2)}(\g)\to \Kill(\g/\g^{(i+1)}).\qquad\qquad\hfill\qed \]
\end{lem}

This shows that the quotients $\Kill(\g)/\Kill^{(i+2)}(\g)$ are encoded in nilpotent quotients of $\g$.

\begin{prop}\label{met3ni}
If $\g$ is metabelian (i.e., $D\g=[\g,\g]$ is abelian), then $\Kill^{(5)}(\g)=\{0\}$. In particular, the surjection $\g\to\g/\g^{(4)}$ induces an isomorphism $\Kill(\g)\to\Kill(\g/\g^{(4)})$.
\end{prop}
\begin{proof}
Denoting $\equiv$ the equivalence relation on $S^2\g$ meaning equality in $\Kill(\g)$, we have, for any elements $x,y,z,u,v\in\g$
\[x\cc [y,[z,[u,v]]]\equiv [x,y]\cc [z,[u,v]]\equiv z\cc [[u,v],[x,y]]=0.\]
The second statement follows by applying Lemma \ref{killni} with $i=3$.
\end{proof}

\begin{cor}
If $\g$ is a regular quadrable metabelian Lie algebra over a field, then it is 3-nilpotent.
\end{cor}
\begin{proof}
Indeed, Proposition \ref{met3ni} shows that $\g^{(4)}$ belongs to the kernel of every invariant symmetric bilinear form.
\end{proof}

\begin{cor}\label{meta3ni}Given a field $K$, we have the following equivalent statements:
\begin{enumerate}[(i)] 
\item\label{supi} Every 3-nilpotent Lie algebra over $K$ has a zero reduced Koszul map;
\item every metabelian Lie algebra over $K$ has a zero reduced Koszul map.
\end{enumerate}
\end{cor}
\begin{proof}
The reverse implication is trivial since 3-nilpotent implies metabelian. 
Conversely, assuming (\ref{supi}), let by contradiction $\g$ be a metabelian Lie algebra with a nonzero reduced Koszul map; by Lemma \ref{somq}, we can suppose, replacing $\g$ by a quotient if necessary, that $\g$ is regular quadratic, say with non-degenerate scalar product $\langle\cdot,\cdot\rangle$. Since $\g$ is metabelian, we have $\langle \g,[[\g,\g],[\g,\g]]\rangle=0$. By invariance, we deduce successively $\langle [\g,[\g,\g]],[\g,\g]\rangle=0$ and $\langle [\g,[\g,[\g,\g]]],\g\rangle=0$. By non-degeneracy, we deduce $[\g,[\g,[\g,\g]]]=0$, i.e.\ $\g$ is 3-nilpotent, in contradiction with (\ref{supi}).
\end{proof}

\begin{rem}
Similar arguments show that if $\g$ is quadrable and $[\g^{(i)},\g^{(j)}]=\{0\}$ for some $i,j\ge 1$, then $\g$ is $(i+j-1)$-nilpotent. For instance, if $\g$ is abelian-by-(2-nilpotent) then $\g$ is 5-nilpotent.

On the other hand, if $\g$ is quadrable and center-by-metabelian (i.e.\ $[\g,D^2\g]=\{0\}$), then $\g$ does not need be 4-nilpotent, nor even nilpotent at all: for instance the well-known 4-dimensional ``split oscillator algebra" with basis $(e,x,y,z)$ and nontrivial brackets $[e,x]=x$, $[e,y]=-y$, $[x,y]=z$ is quadrable, center-by-metabelian (and (2-nilpotent)-by-abelian as well) but is not nilpotent. We can also easily find center-by-metabelian nilpotent quadrable Lie algebras of arbitrary large nilpotency length, e.g.\ the Lie algebras $\mk{w}(2n-1)$ defined in Section~\ref{small}.
\end{rem}

\subsection{Killing modules and direct products}

The case of abelian Lie algebras show that of course the Killing module does not behave well under direct products. However, this is, in a sense, the only obstruction, as the proposition below indicates.

Let $\g_1,\g_2$ be Lie algebras. Then for $k=1,2$, the natural Lie algebra homomorphisms $\g_k\to\g_1\times\g_2\to\g_k$ (whose composite is the identity) induce natural maps $\Kill(\g_k)\to\Kill(\g_1\times\g_2)\to\Kill(\g_k)$, whose composite is the identity, and therefore we obtain natural maps
\begin{equation}\label{kikiki}\Kill(\g_1)\oplus\Kill(\g_2)\to\Kill(\g_1\times\g_2)\to\Kill(\g_1)\oplus\Kill(\g_2),\end{equation}
whose composite is the identity; in particular the left-hand map is injective and the right-hand map is surjective. Clearly the same holds, by restriction, for $\Kill^{(i)}$.

\begin{prop}
For every $i\ge 3$, the above maps restrict to inverse isomorphisms
\[
\Kill^{(i)}(\g_1\times\g_2)\rightleftarrows\Kill^{(i)}(\g_1)\oplus\Kill^{(i)}(\g_2).\]
Moreover, the image of the reduced Koszul map splits according to this product.
\end{prop}
\begin{proof}
It is enough to check that the right-hand map in (\ref{kikiki}) is injective on $\Kill^{(3)}(\g_1\times\g_2)$. Indeed, the image of both $(\g_1\times\{0\})\otimes (\{0\}\times[\g_2,\g_2])$ and $([\g_1,\g_1]\times\{0\})\otimes (\{0\}\times\g_2)$ in $\Kill(\g_1\times\g_2)$ is obviously zero.

For the second statement, it is convenient to view the reduced Koszul map as defined on the 3-cycles (rather than on the 3-homology). Indeed we have a canonical decomposition $Z_3(\g_1\times\g_2)=Z_3(\g_1)\oplus Z_3(\g_2)$, and clearly $\eta(Z_3(\g_k))$, for each of $k=1,2$, is contained in the factor $\Kill^{(3)}(\g_k)$ of $\Kill^{(3)}(\g_1\times\g_2)$.
\end{proof}

\begin{rem}\label{etaprod}The behavior for infinite products is more of a problem. If $\g_n$ is a sequence of Lie algebras over $R$ (the reader can assume that $R=\mathbf{C}$), there is an obvious homomorphism $\Kill(\prod_n\g_n)\to\prod_n\Kill(\g_n)$. It can fail to be surjective (e.g., when all $\g_n$ are abelian of unbounded dimension). The injectivity can also fail (say, when the $\g_n$ are free Lie algebras of unbounded rank). Then it is natural to wonder whether it can happen that the reduced Koszul map of each $\g_n$ vanishes, but not that of $\prod_n\g_n$, or more generally whether the class of Lie algebras with zero reduced Koszul map is stable under taking projective limits of surjective maps.
\end{rem}

\subsection{Change of scalars}

Let $R$ be a commutative ring and let $S$ be a commutative $R$-algebra. We write $\Kill_R$ instead of $\Kill$, $\Lambda^n_R$ instead of $\Lambda^n$, etc., when we need to emphasize the ground ring $R$.

\begin{lem}\label{tenso}
Let $\g$ be a Lie algebra over $R$ and $\g_S=S\ot_R\g$.
We have a natural $S$-module identification $\Lambda^n_S(\g_S)\simeq S\ot\Lambda^n_R(\g)$. If moreover $S$ is flat over $R$, then we also have identifications $\Kill_S(\g_S)\simeq S\ot_R\Kill_R(\g)$ and $H_n^S(\g_S)\simeq S\ot_R H_n^R(\g)$, and the Koszul maps ($\check{\eta}$, $\eta$, $\bar{\eta}$) of the Lie $S$-algebra $\g_S$ are obtained from those of $\g$ by tensoring with $S$.

In particular, if $\bar{\eta}$ vanishes for $\g$ then it vanishes for the $S$-algebra $\g_S$, and if $S$ is faithfully flat over $R$ (e.g., $R$ is a field), then the converse holds.  
\end{lem}
\begin{proof}
Inverse isomorphisms between $\Lambda^n_S(\g_S)$ and $S\ot\Lambda^n_R(\g)$ are given, say for $n=2$, by $(s_1x_1\we s_2x_2)\mapsto s_1s_2\ot (x_1\we x_2)$ from the left to the right and $s\ot (x_1\we x_2)\mapsto (sx_1)\we x_2$ from the right to the left; they restrict (using flatness) to inverse isomorphisms between $Z_n^S(\g_S)$ and $S\ot Z_n^R(\g)$. By right-exactness of the functor $S\ot_R-$, the identification $H_n^S(\g_S)\simeq S\ot_R H_n^R(\g)$ follows. The argument for the other maps is similar. 

The second statement is immediate.
\end{proof}

An easy example of application of Lemma \ref{tenso} is the following:

\begin{lem}\label{reducomp}
Let $\g$ be a finite-dimensional nilpotent Lie algebra over a field $K$ of characteristic zero, whose reduced Koszul map is nonzero (resp.\ is surjective). Then there exists a complex Lie algebra $\g'$ of the same dimension and same nilpotency length, whose reduced Koszul map is nonzero (resp.\ surjective) as well.
\end{lem}
\begin{proof}
Let $L\subset K$ be a countable field of definition of $\g$ (e.g., the field generated by the structural constants in a given basis): this means that $\g\simeq K\ot_L\mk{h}$ for some Lie algebra $\mk{h}$ over $L$. Fix an embedding of $L$ into $\mathbf{C}$. Then the complex Lie algebra $\g'=\mathbf{C}\ot_L\mk{h}$ has the same dimension as $\g$, the same nilpotency length, etc. By Lemma \ref{tenso}, its reduced Koszul map has the same linear algebra properties (dimension of the image and of the cokernel), whence the conclusion.
\end{proof}

\section{2-nilpotent Lie algebras}

The ground ring $R$ is commutative.

\begin{thm}\label{reki4}
Let $\g$ be a Lie algebra over $R$. Suppose that $\mathrm{Tor}_1^R([\g,\g],\g/[\g,\g])=0$. Then the image of $3\hat{\eta}$ is contained in $\Kill^{(4)}(\g)$. In particular, if 3 is invertible in $R$, then the image of the reduced Koszul map $\bar{\eta}$ is contained in $\Kill^{(4)}(\g)$.
\end{thm}

\begin{ex}
The assumption $\mathrm{Tor}_1^R([\g,\g],\g/[\g,\g])=0$ holds when either $[\g,\g]$ is flat or $\g/[\g,\g]$ is a flat $R$-module. In particular, it holds when $R$ is a PID and $[\g,\g]$ is a torsion-free $R$-module.
\end{ex}

\begin{rem}
The proof shows more generally that the theorem holds as soon as the inclusion of $[\g,\g]$ into $\g$ induces an injection after tensoring with $\g/[\g,\g]$. In particular, it holds when $[\g,\g]$ is a direct factor in $\g$.
\end{rem}

\begin{lem}\label{esm}
Consider an exact sequence of $R$-modules
\[0\to W\to V\stackrel{p}\to B\to 0.\]
Suppose that $\mathrm{Tor}_1^R(W,B)=\{0\}$. Then the sequence
\[W\ot W\stackrel{i}\to (V\ot W)\oplus (W\ot V) \stackrel{j}\to V\ot V\]
\[i(w\ot w')=(w\ot w',-w\ot w');\quad j((v\ot w),(w'\ot v'))=v\ot w+w'\ot v';\]
is exact.
\end{lem}
\begin{proof}
The composite map is obviously zero. To prove it is exact, consider an element in the kernel of $j$; we denote it as $((v\ot w),(w'\ot v'))$ as a shorthand to mean that it is a finite sum of such elements. By assumption, we have, in $V\ot V$, $v\ot w+w'\ot v'=0$. In particular, the image in $B\ot V$ of $v\ot w$ is zero. Using the Tor-assumption, we have a monomorphism $B\ot W\to B\ot V$; it follows that the image in $B\ot W$ of $v\ot w$ is zero. Using right-exactness of the functor $-\ot W$, we deduce that $v\ot w$ belongs to the image of $W\ot W$ in $V\ot W$, say $v\ot w=w''\ot w'''$ (meaning a sum of tensors again) in $V\ot W$ (and hence in $V\ot V$). It follows that $((v\ot w),(w'\ot v'))=i(w''\ot w''')$.
\end{proof}

\begin{lem}\label{projvw}
Under the assumptions of Lemma \ref{esm}, there is a unique homomorphism $f$ from the image $(V\ot W)+(W\ot V)$ of $j$ onto $B\ot W$ mapping $v\ot w$ to $p(v)\ot w$ and $w\ot v$ to 0 for all $v\in V$, $w\in W$. 
\end{lem}
\begin{proof}
Such a surjective homomorphism can obviously be uniquely defined at the level of $(V\ot W)\oplus (W\ot V)$, and vanishes on the image of $i$; by Lemma \ref{esm}, this image is exactly the kernel of $j$; hence this homomorphism factors through the image of $j$.
\end{proof}

\begin{lem}\label{lavvw}
Keep the assumptions of Lemma \ref{esm}. There exist unique homomorphisms $g,h$ from the image of $V\ot W$ in $S^2V$ and $\Lambda^2V$ respectively, to $B\ot W$, mapping $v\cc w$ (resp.\ $v\we w$) to $p(v)\ot w$ for all $v\in V$, $w\in W$.
\end{lem}
\begin{proof}
Let us prove the case of $\Lambda^2V$, the proof in the symmetric case being analogous. Define $\lambda:\Lambda^2V\to V^{\ot 2}$ by $\lambda(v\we v')=v\ot v'-v'\ot v$. Then $\lambda$ maps the image $I$ of $V\ot W$ in $\Lambda^2V$ into $V\ot W+W\ot V$. Hence, given $f$ as provided by Lemma \ref{projvw}, the composite homomorphism $f\circ \lambda$ is well-defined on $I$, mapping $v\we w$ to $p(v)\ot w$ for all $v\in V$ and $w\in W$.
\end{proof}

\begin{proof}[Proof of Theorem \ref{reki4}]
Denote by $p$ the projection $\g\to\g/[\g,\g]$. Consider a 3-cycle, given as $x\we y\we z$ (i.e., a sum of such elements). By definition, we have $x\we [y,z]+y\we [z,x]+z\we [x,y]=0$. Applying the function $h$ of Lemma \ref{lavvw} (which holds by the Tor-assumption), we deduce that $p(x)\ot [y,z]+p(y)\ot [z,x]+p(z)\ot [x,y]=0$ in $(\g/[\g,\g])\ot [\g,\g]$. By right-exactness of $-\ot [\g,\g]$, we deduce that $x\ot [y,z]+y\ot [z,x]+z\ot [x,y]\in \g\ot [\g,\g]$ belongs to the image of $[\g,\g]\ot [\g,\g]$. It follows that in $\g\cc\g$, the element $x\cc [y,z]+y\cc [z,x]+z\cc [x,y]$ belongs to the image of $[\g,\g]\cc [\g,\g]$. Thus the image in $\Kill(\g)$ of $3x\cc [y,z]$ belongs to the image of $[\g,\g]\cc [\g,\g]$; in other words, $3\eta(x\we y\we z)\in\Kill^{(4)}(\g)$.
\end{proof}

\begin{prop}\label{nonredu}
Let $R$ be a commutative ring in which 2 is invertible, and suppose that $R$ is not reduced (i.e., contains a nonzero nilpotent element). Then there exists a 2-nilpotent Lie algebra over $R$ whose reduced Koszul homomorphism is nonzero.  
\end{prop}
\begin{proof}
By assumption, there exists an element $t\in R\smallsetminus\{0\}$ such that $t^2=0$. Consider the Lie algebra $\g$ over $R$ whose underlying module is the free module $R^3$ with basis $(e_i)_{i\in\mathbf{Z}/3\mathbf{Z}}$, and whose law is given by the nonzero brackets $[e_i,e_{i+1}]=te_{i+2}$ for $i$ (modulo 3). It is readily seen to be indeed a 2-nilpotent Lie algebra.

Consider the element $\hat{c}=e_1\ot (e_2\we e_3)\in\g\ot\Lambda^2\g$, and let $c$ be its image in $\Lambda^3\g$. We will prove the proposition by showing that $c\in Z_3(\g)\smallsetminus \Ker(\eta)$.

Since 2 is invertible, we have $e_i\we e_i=0$ for all $i$. Thus we have $\bd_3(c)=\sum_{i\in\mathbf{Z}/3\mathbf{Z}} x_i\we tx_i=0$, so that $c\in Z_3(\g)$.

To check that $c\notin\Ker(\eta)$, let us define a grading of $\g$ in the abelian group $(\mathbf{Z}/2\mathbf{Z})^3$ with basis $(\omega_i)_{i\in\mathbf{Z}/3\mathbf{Z}}$, for which $e_i$ has weight $\omega_{i+1}+\omega_{i+2}$. Then $c$ has weight (0,0,0). So let us compute $\Kill(\g)_0$. We have $(S^2\g)_0=\bigoplus_i\g_i\cc\g_i$, a free module of rank 3 with basis $(e_i\cc e_i)_{i\in\mathbf{Z}/3\mathbf{Z}}$. Also $((\Lambda^2\g)\ot\g)_0$ is the free module on the basis $((e_i\we e_{i+1})\ot e_{i+2})_{i\in\mathbf{Z}/3\mathbf{Z}}$. We have $T((e_i\we e_{i+1})\ot e_{i+2})=t(e_i\cc e_i-e_{i+1}\cc e_{i+1})$. Thus, in the above basis of $(S^2\g)_0$, the image of $T$ is the set of triples $(u,v,w)\in (tR)^3$ such that $u+v+w=0$. On the other hand, we have $\check{\eta}(\hat{c})=e_0\cc [e_1,e_2]=te_1\cc e_1$, which is $(t,0,0)$ in the above basis, and thus does not belong to the image of $T$. Thus $\bar{\eta}(c)\neq 0$.
\end{proof}

Let us now indicate a counterexample over a field of characteristic 3.

We begin with a general classical definition: let $R$ be an arbitrary commutative ring and let $V,W$ be $R$-modules with an alternating bilinear map $f:V\times V\to W$. Define a Lie algebra structure on $M\oplus N$ by $[(v,w),(v',w')]=(0,f(v,v'))$, and denote it by $\g(V,W,f)$. Note that it is 2-nilpotent, with a natural grading in $\{1,2\}$.

Now we choose $V=V(R)$ to be the free $R$-module of rank 7, with basis $(e_i)_{-3\le i\le 3}$; we endow it with the resulting grading in $\{-3,\dots,3\}\subset\mathbf{Z}$. We choose $W=V$, and define $f$ as follows:
\begin{center}
\begin{tabular}{c||c|c|c|c|c|c|c|}
  & $e_{-3}$ & $e_{-2}$ & $e_{-1}$ & $e_0$ & $e_1$ & $e_2$ & $e_3$ \\
  \hline\hline
$e_{-3}$ &  &  &  & $-e_{-3}$ & $e_{-2}$ & -$e_{-1}$ & $e_{0}$\\
  \hline
$e_{-2}$ &  &  & $-e_{-3}$ & $e_{-2}$ &  & $-e_{0}$ & $e_{1}$\\
  \hline
  $e_{-1}$ &  & $e_{-3}$ &  & $e_{-1}$ & $-e_{0}$ &  & $-e_{2}$\\
  \hline
  $e_{0}$ & $e_{-3}$ & $-e_{-2}$ & $-e_{-1}$ &  & $e_{1}$ & $e_{2}$ & $-e_{3}$\\
  \hline
  $e_{1}$ & $-e_{-2}$ &  & $e_{0}$ & $-e_{1}$ &  & $e_{3}$ & \\
  \hline
  $e_{2}$ & $e_{-1}$ &  $e_{0}$ & & $-e_{2}$ & $-e_{3}$ & &\\
  \hline
  $e_{3}$ & $-e_{0}$ & $-e_{1}$ & $e_{2}$ & $e_{3}$ & &&\\
  \hline\end{tabular}
\end{center}
(for instance $f(e_1,e_2)=e_3$, $f(e_1,e_3)=0$, etc.) This follows the definition of the ``octonion product" given by R.~Wilson in \cite{Wil}.
In addition, define a symmetric bilinear form $b$ on $V$ by $b(e_i,e_j)=1$ if $i+j=0$ and 0 otherwise. As mentioned in \cite{Wil}, the trilinear form $t$ defined by $t(x,y,z)=b(f(x,y),z)$ is alternating.

Although this is not important here, note that $(V(R),f)$ is a Lie algebra when $3=0$ in $R$, but not otherwise. 

We now consider the 2-nilpotent Lie algebra $\g(V,V,f)$. The symmetric bilinear form defined by $B((v_1,v_1'),(v_2,v'_2))=b(v_1,v'_2)+b(v_2,v'_1)$ is thus invariant. Define $E_i=(e_i,0)$ and $F_i=(0,e_i)$, so that $\g(V,V,f)$ is a free module of rank 14 over the basis $((E_i)_{-3\le i\le 3},(F_i)_{-3\le i\le 3})$

Now consider the 3-chain
\[c=E_0\we (E_1\we E_{-1}+E_2\we E_{-2}+ E_{-3}\we E_3)-E_1\we E_2\we E_{-3}-E_{-1}\we E_{-2}\we E_3.\]
Then $\eta(c)$ is represented by $3e_0\cc f_0-e_1\cc f_{-1}-e_{-1}\cc f_1$, and thus $B(\eta(c))=1$, so $\eta(c)\neq 0$ in $\Kill(\g)$.

On the other hand, $\bd_3(c)=3e_0\we f_0$. In particular, if $3=0$ in $R$ then $c$ is a 3-cycle and hence the reduced Koszul map of $\g(V,V,f)$ is nonzero. We have proved:

\begin{prop}\label{char3nonz}
For every field of characteristic 3, there exists a 2-nilpotent Lie algebra whose reduced Koszul map is nonzero.
\end{prop}

\section{Vanishing for graded Lie algebras}\label{vagra}

\begin{thm}\label{thmkos}
Assume that 6 is invertible in the ground commutative ring. Let $\g$ be graded in an abelian group $A$ endowed with an injective group homomorphism $\phi$ into a subfield of the ground ring. Then the Koszul homomorphism of $\g$ is concentrated in degree 0.
\end{thm}
\begin{proof}
Since 2 is invertible, we have natural embeddings of both $\g\we\g$ and $\g\cc\g$ into $\g\ot\g$. Let $s$ be the flip of $\g\ot\g$ (namely $s(x\ot y)=y\ot x$).

Define $\bd'_3,\bd''_3:\g\we\g\we\g\to \g\ot\g$ by 
\[\bd'_3(x\we y\we z)=x\ot [y,z]+y\ot [z,x]+z\ot [x,y];\qquad\bd''_3=s\circ \bd'_3.\] 

Also define graded module endomorphisms $\Phi,\Psi:\g\ot\g\to\g\ot\g$ as being the multiplication by $\phi(j-i)$, resp.\ $\phi(i+j)$ on $\g_i\ot\g_j$.

Define $T',T'':\g\ot\g\ot\g\to\g\ot\g$ by
\[T'(x\ot y\ot z)=x\ot [y,z]-y\ot [z,x],\qquad T''=s\circ T'.\] Let $\iota$ be the projection $\g^{\ot 3}\to\Lambda^3\g$.

We claim that $3\Phi(\bd'_3-\bd''_3)\iota=\Psi(\bd'_3+\bd''_3)\iota+(T'+T'')U$ for some operator $U:\g\ot\g\ot\g\to\g\ot\g$ respecting the grading in $A$.

Indeed, for all $i,j,k\in A$ and all $x_i\in\g_i$, $y_j\in\g_j$, $z_k\in\g_k$, we have, defining $c=x_i\we y_j\we z_k$ (for convenience, we view $\phi$ as an inclusion)
\begin{align*}
3\Phi \bd'_3(c) =& 3\Phi(x_i\ot [y_j,z_k]+y_j\ot [z_k,x_i]+z_k\ot [x_i,y_j])\\
=& (j+k-i)3x_i\ot [y_j,z_k]\\ &+(k+i-j)3y_j\ot [z_k,x_i]+(i+j-k)3z_k\ot [x_i,y_j];\\
\Psi \bd'_3(c)= & (i+j+k)(x_i\ot [y_j,z_k]+y_j\ot [z_k,x_i]+z_k\ot [x_i,y_j]);\\
(3\Phi -\Psi) \bd'_3(c) = & 2\big((j+k-2i)x_i\ot [y_j,z_k]\\ &+(k+i-2j)y_j\ot [z_k,x_i]+(i+j-2k)z_k\ot [x_i,y_j]\big);
\end{align*}
writing $y_j\ot [z_k,x_i]=x_i\ot [y_j,z_k]-T'(x_i\ot y_j\ot z_k)$ and
$z_k\ot [x_i,y_j]=x_i\ot [y_j,z_k]+T'(z_k\ot x_i\ot y_j)$, we deduce
\[(3\Phi -\Psi) \bd'_3(c) = -2(k+i-2j)T'(x_i\ot y_j\ot z_k)+2(i+j-2k)T'(z_k\ot x_i\ot y_j);\]
hence if we define $U(x_i\ot y_j\ot z_k)=-2(k+i-2j)x_i\ot y_j\ot z_k+2(i+j-2k)z_k\ot x_i\ot y_j$, we have proved that $(3\Phi \bd'_3-\Psi \bd'_3)\iota=T'U$.
Note that $s\Phi=-\Phi s$ and $s\Psi=\Psi s$. Hence
\[T''U=sT'U=s(3\Phi -\Psi) \bd'_3\iota=-(3\Phi +\Psi) s \bd'_3\iota=-(3\Phi +\Psi)\bd''_3\iota;\]
thus $(3\Phi (\bd'_3-\bd''_3)-\Psi (\bd'_3+\bd''_3))\iota=(T'+T'')U$, and the claim is proved.

 Write $\hat{\bd_3}=\bd'_3-\bd''_3$ and $\hat{\eta}=\bd'_3+\bd''_3$. The previous equality can be rewritten as $(3\Phi\hat{\bd_3}-\Psi\hat{\eta})\iota=(T'+T'')U$. Consider the projections $\g^{\ot 2}\stackrel{p}\to S^2\g\stackrel{\pi}\to\Kill(\g)$.
  
Clearly $\pi p$ vanishes on $\mathrm{Im}(T')+\mathrm{Im}(T'')$, so we deduce $\pi p(3\Phi\hat{\bd_3}-\Psi\hat{\eta})\iota=0$. Since $\iota$ is surjective, it follows that $\pi p(3\Phi\hat{\bd_3}-\Psi\hat{\eta})=0$.

Denoting
\[\check{\eta}'(x\we y\we z)=x\cc [y,z]+y\cc [z,x]+z\cc [x,y],\]
we obviously have $\check{\eta}'=p\hat{\eta}$. Also note that we can view $\Psi$ (unlike $\Phi$) as a graded self-operator on every graded module, given by scalar multiplication by $\phi(i)$ on the $i$-th component. Thus defined, it commutes with all graded homomorphisms. Thus $\pi p\Psi\eta=\Psi\pi\check{\eta}'$; hence $3\pi p\Phi\hat{\bd_3}=\Psi\pi\check{\eta}'$. Since $\pi\check{\eta}'=3\eta$, this can be rewritten as
\begin{equation}3\pi p\Phi\hat{\bd_3}=3\Psi\eta.\label{3psi}\end{equation}

Now let $c$ be a 3-cycle. Denoting by $p'$ the projection $\g^{\ot 2}\to\Lambda^2\g$, we have $p'\hat{\bd_3}=2\bd_3$, and since the projection $p'$ is injective on antisymmetric tensors and hence on the image of $\hat{\bd_3}$, we deduce that $\hat{\bd_3}(c)=0$.
Using (\ref{3psi}), it follows that $3\Psi\eta(c)=0$. If $c$ is homogeneous of nonzero weight $\alpha$, the latter is equal to $3\phi(\alpha)\eta(c)$. Since $3\phi(\alpha)$ is invertible in the ground ring, we deduce that $\eta(c)=0$.
\end{proof}

\begin{cor}\label{nstar2}
If the ground ring is a commutative $\mathbf{Q}$-algebra and $\g$ admits a grading in $\mathbf{N}^*$, then the reduced Koszul map of $\g$ vanishes.

If $p=3m+1\ge 5$ is prime, the ground ring is a commutative $\mathbf{Z}/p\mathbf{Z}$-algebra and $\g$ admits a grading in $\{1,2,\dots,m\}$, then the reduced Koszul map of $\g$ vanishes.
\end{cor}
\begin{proof}
In both cases, we obtain a grading in the ground ring, for which 0 is not the sum of 2 weights; hence $(S^2\g)_0=0$ and thus the only weight left by Theorem \ref{thmkos} is discarded. 
\end{proof}

\begin{rem}
Let $K$ be any field. Then $\mk{sl}_2(K)$ admits a grading in $\mathbf{Z}/2\mathbf{Z}$ for which the reduced Koszul map is nonzero in degree 1, namely the grading induced by the involution $x\mapsto -x^{\intercal}$. More generally (say for $K=\mathbf{C}$), a cyclic permutation of the factors in $\mk{sl}_2(K)^n$ provides a grading in $\mathbf{Z}/n\mathbf{Z}$ for which the reduced Koszul map is nonzero in every degree. Thus, in Theorem \ref{thmkos} when $R$ is a $\mathbf{Q}$-algebra, the fact that the grading is in a torsion-free abelian group is essential.
\end{rem}

\begin{proof}[Proof of Corollary \ref{rcomgr}]
We follow the proof of Theorem \ref{vagra}, but we need to carefully redefine the involved maps, especially their domain of definition and target space.
For $k\ge 1$, define $\overline{\g^{\ot k}}$ as the submodule of $\prod_{n\in\mathbf{Z}}(\g^{\ot k})_n$ consisting of sequences with support bounded below. Note that it naturally contains $\overline{\g}^{\ot k}$ (with equality for $k=1$).
Then $\Phi,\Psi$ (as defined in the proof of Theorem \ref{vagra}) are well-defined as self-operators of $\overline{\g^{\ot 2}}$, and $U$ is well-defined as an operator from $\overline{\g^{\ot 3}}$ to $\overline{\g^{\ot 2}}$. Also, since $T,T''$ are graded operators, they naturally extend to operators from $\overline{\g^{\ot 3}}$ to $\overline{\g^{\ot 2}}$. Eventually, the formula $3\pi p\Phi\hat{\bd_3}=3\Psi\eta$ is proved along the same lines, and we also obtain the conclusion that $\eta(c)=\eta(c_0)$ for every 3-cycle $c=\sum_{n\in\mathbf{Z}} c_n$ with $c_n\in Z_3(\g)_n$.
\end{proof}

I do not know if Corollary \ref{rcomgr} can directly follow from Theorem \ref{vagra} (see the closely related Remark \ref{etaprod}).

\section{A few families and small-dimensional nilpotent quadrable Lie algebras}\label{small}

For the moment, we work over an algebraically closed field $K$ of characteristic $\neq 2$.

\subsection{The regular quadratic Lie algebras $\mk{w}(\lambda)$} We recall a general well-known construction, appearing in \cite{FS}, which provides most of the small-dimen\-sional regular quadratic nilpotent Lie algebras: namely those double extensions (see \S\ref{rqla}) of an orthogonal space (viewed as an abelian regular quadratic Lie algebra).

Since the ground field is algebraically closed, the orthogonal spaces are determined, up to isomorphism, by their dimension $n$, while the nilpotent skew-symmetric endomorphisms are classified, up to orthogonal conjugation, by partitions of $n$ in which every even number occurs with even multiplicity (see \cite{CM}). It is convenient here to rather consider the equivalent data of partitions of $n$ in which no number congruent to 2 modulo 4 occurs (gathering the even numbers pairwise). Given such a partition of $n$, given as $n=a_1+\dots+a_p+b_1+\dots+b_q$, with each $a_i$ odd and each $b_j$ multiple of 4, a representative for the pair $(V,D)$ is given by a space with basis
\[ e_{ij},f_{k\ell}, \quad 1\le i\le p,\quad 1 \le j\le a_i,\quad 1\le k\le q,\quad 1\le \ell\le b_{q},\] and $D$ defined by
\[De_{ij}=(-1)^ie_{i,j+1}, \quad 1\le j\le a_i-1; \quad Df_{k,\ell}=(-1)^{k}f_{k,\ell+2}, \quad 1\le \ell\le b_{k}-2,\] and $De_{i,a_i}=Df_{k,b_{k}}=Df_{k,b_{k}-1}=0$;
and the nonzero scalar products on the basis are
\[\la e_{ij},e_{i,a_i+1-j}\ra=\la f_{k\ell},f_{k,b_k+1-\ell}\ra=1.\]
Thus the nonzero brackets are, denoting $x=(1,0,0)$ and $z=(0,0,1)$ 
\[[x,e_{ij}]=(-1)^ie_{i,j+1},\quad [x,f_{ij}]=(-1)^if_{i,j+2}\]
\[ [e_{ij},e_{i,a_i-j}]=(-1)^iz,\quad 1\le j\le a_i-1; \quad  [f_{k\ell},f_{k,b_k-1-\ell}]=(-1)^iz,\quad 1\le \ell\le b_k-2\]

 We denote the corresponding Lie algebra as $\mk{w}(a_1\oplus\dots \oplus b_\ell)$; its nilpotency length is $\max(\max_ia_i,\max_k b_k/2)$; also we write $[k]n$ for $n\oplus\dots\oplus n$ ($k$ times).

The $(kn+2)$-dimensional Lie algebra $\g=\mk{w}([k]n)$ is Carnot with Carnot-grading given as follows: all basis elements are homogeneous, $x$ has degree 1, and:
\begin{itemize}
\item if $n$ is odd: $e_{ij}$ has degree $j$, $z$ has degree $n$;
\item if $n$ is even: $e_{ij}$ has degree $\lceil j/2\rceil$, $z$ has degree $n/2$.
\end{itemize}

For an arbitrary partition, the Lie algebra is not necessarily Carnot; still it admits several gradings. The simplest one is in $\{0,1,2\}$, with $x$ in degree 0, all $e_{ij}$ in degree 1, and $z$ in degree 2.

There also exist gradings in $\mathbf{N}^*$, given as follows: fix a large enough integer $r$ (namely $\ge a_i-1$ and $\ge b_k/2-1$ for all $i,k$):
\begin{itemize}
\item let $x$ have degree 2 and $z$ have degree $2r$;
\item let $e_{ij}$ have degree $r-a_i+2j$;
\item let $f_{k\ell}$ have degree $r-b_k/2+2\lceil\ell/2\rceil$.
\end{itemize}

The regular quadratic Lie algebras $\mk{w}(\lambda)$, where $\lambda$ is a partition of $n$ as above (with no integer congruent to 2 modulo 4), are characterized in \cite{MPU} as the nilpotent ones for which the ``dup number" is nonzero: this number reflects a degeneracy property of the alternating 3-form associated to the choice of bilinear form (they show it is actually a property of the underlying Lie algebra). 

\subsection{The regular quadratic Lie algebras $\mk{X}(3m-1)$ and $\mk{Y}(3m)$}

The smallest complex nilpotent quadrable Lie algebra not among the $\mk{w}(\lambda)$ (``deeper" in the terms of \cite{Pel}, ``ordinary" in the terms of \cite{MPU}) has dimension 8.

We introduce, for convenience, the Lie algebra $\mkh$ with basis $(X_{n,i})$, where $n$ ranges over $\mathbf{Z}$ and $i$ ranges over $\{0,1,-1\}$; the law being given by the nonzero brackets
\[[X_{n,i},X_{m,j}]=(i-j)X_{n+m,i+j},\quad m,n\in \mathbf{Z},\;(i,j)\in\{0,1,-1\}^2\smallsetminus\{\pm(1,1)\}.\]

(The reader may recognize a current algebra of $\mk{sl}_2$.) Given $r\in\mathbf{Z}$, define a symmetric bilinear bracket by the nonzero scalar products
\[\langle X_{n,i},X_{m,j}\rangle_r=(1-3|i|)\delta_{n+m,r}\delta_{i+j,0}.\]

Then it is invariant: indeed, we have 
\[\langle X_{n,i},[X_{m,j},X_{p,k}]\rangle_r=(j-k)(1-3|i|)\delta_{n+m+p,r}\delta_{i+j+k,0},\]
and we need to check that the right-hand term is invariant under cyclic permutations of $(i,j,k)$. If $i+j+k\neq 0$ this is clear. Otherwise, we have two possibilities: either $(i,j,k)=(0,0,0)$ and again this is clear, or $\{i,j,k\}=\{0,1,-1\}$. We see that if $(i,j,k)$ is a cyclic permute of $(0,1,-1)$ (resp.\ of $(0,-1,1)$), then $(j-k)(1-3|i|)$ is equal to 2 (resp.\ -2). Hence the scalar product is invariant. Note that $(n,i)$ provides a grading in $\mathbf{Z}^2$.

If we rewrite $X_{n,i}=T_{3n+i}$, we see that $(T_n)$ is a basis, and that $[T_n,T_m]=e(m,n)T_{n+m}$, with $e$ the skew-symmetric map factoring through $(\mathbf{Z}/3\mathbf{Z})^2$ characterized by: $e(1,0)=1$, $e(1,-1)=2$, $e(-1,0)=-1$. Moreover, we have $\la T_n,T_m\ra_{r}=\ell(n)\delta_{n+m,3r}$, where $\ell(n)=1$ if $n\in 3\mathbf{Z}$ and $\ell(n)=2$ otherwise.

Note that if we consider the subalgebra $\mkh_{\ge 1}$ generated by the $T_n$ for $n\ge 1$, the kernel of the scalar product $\la\cdot,\cdot\ra_{k}$ is the ideal $\mkh_{\ge 3k}$ with basis the $T_n$ for $n\ge 3k$. In particular, if we define $\mathcal{X}(3k-1)$ as the $(3k-1)$-dimensional quotient $\mkh_{\ge 1}/\mkh_{\ge 3k}$, it is a regular quadratic Lie algebra when endowed with $\la\cdot,\cdot\ra_{3k}$. It is Carnot-graded of nilpotency length $2k-1$, with $T_{3i-2}$ and $T_{3i-1}$ of degree $2i-1$ and $T_{3i}$ of degree $2i$.

Note that $\mk{X}(2)$ is abelian and $\mk{X}(5)\simeq\mk{w}(3)$; this is the only exceptional isomorphism with this family, as it is easily checked that for $3k-1\ge 8$, the Lie algebra $\mk{X}(3k-1)$ admits no nontrivial grading in $\{0,1,2\}$.

(Up to an easy change of variables, the regular quadratic Lie algebra $\mk{X}(3k-1)$ appears in \cite[\S 5.3]{FS}, and reappears as $\mathcal{A}_{1,3k-1}$ in \cite{Pel}, where the solvable Lie algebra $\mk{h}_{\ge 0}/\mkh_{\ge 3k+1}$ is studied as a regular quadratic Lie algebra under the bracket $\la\cdot,\cdot\ra_{k}$.)

Also, if we consider the subalgebra $\mkh_{\ge 2}$, the kernel of the scalar product $\langle\cdot,\cdot\rangle_k$ is the ideal $\mkh_{\ge 3k+1}$. Thus the quotient is a $(3k)$-dimensional regular quadratic Lie algebra which we denote by $\mk{Y}(3k)$, with basis $(T_2,\dots,T_{3k+1})$, which is also $(X_{m,i})$ for $1\le m\le k$ and $-1\le i\le 1$. 
The Lie algebra $\mathfrak{Y}(3k)$ is Carnot-graded of nilpotency length $k$, with $X_{ni}$ of degree $n$ (thus every homogeneous component of the Carnot-grading has dimension 3). We see that $\mk{Y}(3)$ is abelian, $\mk{Y}(6)\simeq\mk{w}(4)$; again, there is no other isomorphism with the $\mk{w}(\lambda)$.

Note that unlike the family of $\mk{w}(\lambda)$ which are all solvable of length 3 (indeed center-by-metabelian), the Lie algebras $\mk{X}(3k-1)$ and $\mk{Y}(3k)$ have unbounded solvability length.

\subsection{Small-dimensional nilpotent quadrable Lie algebras}

Let us now list the nilpotent quadrable Lie algebras of dimension $\le 9$ over the complex numbers (or any algebraically closed ground field of characteristic zero, by a standard argument); we follow Kath \cite{Kat}. (The classification in dimension $\le 7$ is due to Favre and Santharoubane \cite{FS}.) However the classification in \cite{Kat} is over the real numbers and does not make the correspondence with the above Lie algebras explicit, so the list we 
provide looks somewhat different; moreover I had to check that Kath's classification, which is made over real numbers, does not miss any complex quadratic Lie algebra. 

Recall that the class of quadrable Lie algebra is invariant under taking direct products with abelian 1-dimensional Lie algebras, and hence we can stick to those Lie algebras with no such direct product decomposition, or equivalently whose center is contained in the derived subalgebra (we call such Lie algebras {\em essential}). 

For the Lie algebras $\mk{w}(\lambda)$, the maximal abelian direct factors have dimension the number of occurrences of 1 in the partition $\lambda$; thus essential means that 1 does not occur in $\lambda$.

Up to dimension 9, the essential Lie algebras among the $\mk{w}(\lambda)$ are:

\begin{itemize}
\item dimension 5: $\mk{w}(3)$;
\item dimension 6: $\mk{w}(4)$;
\item dimension 7: $\mk{w}(5)$;
\item dimension 8: $\mk{w}(3\oplus 3)$;
\item dimension 9: $\mk{w}(7)$ and $\mk{w}(3\oplus 4)$ (the latter is not Carnot).
\end{itemize}

Now let us describe the remaining essential quadrable Lie algebras, relying on \cite{Kat}.

In dimension 8, the second one is $\mk{X}(8)$ (denoted 6(b) in \cite{Kat}, while $\mk{w}(3\oplus 3)$ is 7(b) there). (Note that it follows from this description that all quadrable nilpotent Lie algebras of dimension $\le 8$ are Carnot.)

In dimension 9, there are 5 isomorphism classes, including the two previous ones (3(a) and 5(c) in \cite{Kat}), and three more. One is $\mk{Y}(9)$ (denoted 7(c) in \cite{Kat}).

The second one (4(c) in \cite{Kat}) is similar to the $\mk{w}(\lambda)$ in the sense that it also admits a grading in $\{0,1,2\}$ with abelian 0-component and central 2-component.

It has a basis $(X_1,\dots,X_9)$ with grading given by the indices, with nonzero brackets $[X_i,X_j]=a_{i,j}X_{i+j}$ given by $a_{1,2}=-1$, $a_{2,3}=1$, $a_{1,3}=1$, $a_{1,6}=-1$, $a_{3,6}=1$,  $a_{2,5}=-1$, $a_{3,5}=1$, $a_{2,7}=-1$, $a_{1,7}=1$, the scalar product being given by $\la X_i,X_j\ra=1$ if $i+j=10$ and 0 otherwise. Its nilpotency length is 5 and actually it has a grading in $\{1,\dots,5\}$ given (in self-explanatory notation) by $12\oplus 3\oplus 456\oplus 7\oplus 89$; another grading is given by the indices. The $\{0,1,2\}$-grading is given by $26\oplus 13579\oplus 48$.
 It is not Carnot because the associated Carnot Lie algebra is not quadrable: indeed its derived subalgebra has codimension 3 while its center has dimension 4 (indeed $X_6$ becomes central, indeed generates a direct factor, in the associated Carnot Lie algebra).

The third example appears as a twisted version of $\mk{w}(7)$ (both appear as 3(a) in \cite{Kat}). To describe it, let us first describe again $\mk{w}(7)$ with convenient notation: it admits the basis $(Y_i)_{i\in\{1,\dots,11\}\smallsetminus\{2,10\}}$, with scalar product $\la Y_i,Y_j\ra=1$ if $i+j=12$ and 0 otherwise, and the nonzero brackets being $[Y_1,Y_i]=(-1)^iY_{i+1}$ and $[Y_i,Y_{11-i}]=(-1)^iY_{11}$ for $3\le i\le 8$. 

The twisted version $\mk{w}(7)^{\textnormal{tw}}$ is defined on the same basis by ``adding" the nonzero brackets $[Y_3,Y_4]=Y_7$, $[Y_3,Y_5]=-Y_8$, and $[Y_4,Y_5]=Y_9$. Then $\mk{w}(7)^{\textnormal{tw}}$ also has nilpotency length 7, but is not isomorphic to $\mk{w}(7)$: a way to see this is by observing that the second derived subalgebra of $\mk{w}(7)^{\textnormal{tw}}$ is 2-dimensional (equal to the center, generated by $Y_9,Y_{11}$), while in $\mk{w}(7)$ it is reduced to the line generated by $Y_{11}$.

\begin{thm}\label{everyni}
For every nilpotent Lie algebra $\g$ over a field of characteristic 0, of dimension $\le 9$, the reduced Koszul map vanishes.
\end{thm}
\begin{proof}
Clearly, we can suppose that $K=\mathbf{C}$ (see Lemma \ref{reducomp}). By Lemma \ref{somq}, we can suppose that $\g$ is quadrable. The way we introduced the above examples shows that in all cases, $\g$ admits a positive grading. Hence we can apply Corollary \ref{nstar}.
\end{proof}

In dimension 10, a case-by-case verification based on the classification in \cite{Kat} seems to also yield the same conclusion: again, every 10-dimensional quadrable Lie algebra admits a grading in the positive integers.

\section{A nilpotent Lie algebra with nonzero reduced Koszul map}\label{12dim}

We construct here a nilpotent Lie algebra with a nonzero reduced Koszul map. 

We fix a field $R=K$ of characteristic not 2 (the same working over any commutative ring in which $2$ is invertible). We use the notion of double extension recalled in \S\ref{rqla}.

We start from the 7-dimensional Lie algebra $\mk{w}(5)$ defined in \S\ref{small}, but rename the basis as $(Y_1,Y_4,Y_5,Y_6,Y_7,Y_8,Y_{11})$, so that the scalar product is defined by: $\la Y_i,Y_j\ra=1$ if $i+j=12$ and 0 otherwise, and the nonzero brackets are:
\begin{equation}[Y_1,Y_i]=(-1)^iY_{i+1},\; 4\le i\le 7;\qquad [Y_4,Y_7]=-[Y_5,Y_6]=Y_{11}.\label{eqw33}\end{equation}

We define a 10-dimensional regular quadratic Lie algebra $\mkh$ as the orthogonal direct product of $\mk{w}(5)$ with a 3-dimensional abelian quadratic Lie algebra with basis $(Z_3,Z_6,Z_9)$ with nonzero scalar products given by $\la Z_3,Z_9\ra=\la Z_6,Z_6\ra=1$. 

We now define a skew-symmetric derivation of $\mkh$ by
\[DY_1=Y_4,\;DY_4=Y_7+Z_3,\;DY_5=-Y_8,\; DY_8=-Y_{11}\]
\[DZ_3=Z_6,\;DZ_6=-Z_9,\; DZ_9=-Y_8,\quad DY_6=DY_7=DY_{11}=0;\]

Finally, we define the 12-dimensional regular quadratic Lie algebra $\g$ as the double extension of $\mkh$ by $D$, in which we write $E_3=e$, $E_9=f$.

Thus $\g$ has the basis $(E_3,E_9,Y_1,Y_4,\dots,Y_8,Y_{11},Z_3,Z_6,Z_9)$ with the nonzero scalar products 
\[\la E_i,E_{12-i}\ra=\la Y_i,Y_{12-i}\ra=\la Z_i,Z_{12-i}\ra=1,\]
the nonzero brackets being those in (\ref{eqw33}) along with
\[[E_3,Y_1]=Y_4,\;[E_3,Y_4]=Y_7+Z_3,\;[E_3,Y_5]=-Y_8,\;[E_3,Y_8]=-Y_{11},\]
\[[E_3,Z_3]=Z_6,\;[E_3,Z_6]=-Z_9,\;[E_3,Z_9]=-Y_8\]
\[[Y_1,Y_8]=[Y_4,Y_5]=[Z_3,Z_6]=[Y_4,Z_9]=E_9\]

We now consider the 3-chains 
\[c_1=E_3\we Y_1\we Y_8,\; c_2=E_3\we Y_4\we Y_5,\; c_3=E_3\we Y_4\we Z_9,\; c_4=Y_1\we Y_4\we Y_7,\]\[\; c_5=E_3\we Z_3\we Z_6,\; c_6=Y_1\we Y_6\we Y_5,\; c_7=Y_1\we Y_4\we Z_3,\; c_8=E_3\we Z_6\we Y_7,\]
\[c=2c_1+4c_2-3c_3+c_4-3c_5-3c_6+4c_7+3c_8.\]
If we define the 2-chains
\[e_{3}=E_3\we E_9,\; y_{1}=Y_1\we Y_{11},\; y_{4}=Y_4\we Y_8,\;y_{5}=Y_5\we Y_7\;, z_{3}=Z_3\we Z_9\] 
and consider the 7-dimensional subspace $V$ of $\Lambda^2\g$ with basis $\mathcal{B}=(e_{3},y_{1},y_{4},y_{5},z_{3},Z_3\we Y_5,Z_9\we Y_7)$, then a computation shows that $\bd_3(c_i)\in V$ for all $i=1,\dots,8$, and the matrix of $\bd_3$ with respect to the bases $(c_1,\dots,c_8)$ and $\mathcal{B}$ is given by 

\[\begin{pmatrix}
1 & 1 & 1 & 0 & 1 & 0 & 0 & 0\\
1 & 0 & 0 & 1 & 0 & 1 & 0 & 0 \\
-1& 1 & 1 & 1 & 0 & 0 & 0 & 0 \\
0 & 1 & 0 &-1 & 0 & 1 & 0 & 0 \\
0 & 0 &-1 & 0 & 1 & 0 & 0 & 0 \\
 0&-1 & 0 & 0 & 0 & 0 & 1 & 0 \\
0 & 0 & 1 & 0 & 0 & 0 & 0 & 1 
      \end{pmatrix}\]

Computation shows that the vector $(2,4,-3,1,-3,-3,4,3)$ belongs to the kernel of this matrix. This means that $\bd_3(c)=0$. 

On the other hand, we have $J(c_i)=1$ for $i\le 6$ and $J(c_7)=J(c_8)=0$. Thus $J(c)=2+4-3+1-3-3=-2\neq 0$.

It follows that $J$ does not vanish on 2-cycles, and thus the reduced Koszul map $\g$ is nonzero: in homology, the 3-homology class $c$ has a nonzero image; in cohomology, the 3-cocycle $J$ is not exact. 

To summarize, let us state:

\begin{thm}
The above 12-dimensional nilpotent Lie algebra $\g$ has a nonzero reduced Koszul map. More precisely, the 3-form associated with the given scalar product is not exact.
\end{thm}

\begin{rem}
It can easily be checked that $\Kill(\g)$ is 5-dimensional. The dual $\Sym^2(\g)^\g$ can be described as the direct sum $W_1\oplus W_2\oplus W_3$, where $W_1$ is the 3-dimensional subspace consisting of symmetric bilinear forms on $S^2(\g/[\g,\g])$, $W_2$ is the line generated by a certain invariant scalar product on the 5-dimensional quotient $\g/\g^{(4)}$ and $W_3$ is the line generated by the scalar product. It is easy to see that the elements not in the hyperplane $W_1+W_2$ are precisely the non-degenerate forms. Also, using that $\g/\g^{(4)}$ is graded in $\mathbf{N}^*$ (or by a simple direct computation), the elements of $W_1+W_2$ have a zero image in $H^3(\g)$: thus the reduced Koszul map has rank~1 as a linear map. 
\end{rem}

\begin{rem}
The Lie algebra $\g$ has a grading in $\mathbf{Z}/4\mathbf{Z}$ given by the indices, namely for which $\g_0$ has basis $\{Y_4,Y_8\}$, $\g_{1}$ has basis $\{Y_1,Y_5,Z_9,E_9\}$, $\g_{-1}$ has basis $\{E_3,Z_3,Y_7,Y_{11}\}$ and $\g_2$ has basis $\{Y_6,Z_6\}$.
\end{rem}

\begin{rem}
The above Lie algebra is center-by-metabelian: its second derived subalgebra, namely the plane generated by $X$ and $X'$, is central.

It seems from \cite{Kat} that all nilpotent complex Lie algebras of dimension $\le 10$ are center-by-metabelian; however the 11-dimensional quadrable Lie algebra $\mk{X}(11)$ is not center-by-metabelian.
\end{rem}

\begin{rem}
In a computer verification, Louis Magnin checked the above computation, and also computed the Betti numbers of $\g$: $(1,2,4,9,15,22,26,22\dots)$. Moreover, he checked that $\g$ is characteristically nilpotent in the sense that every self-derivation is nilpotent; equivalently, $\g$ admits no nontrivial grading in $\mathbf{Z}$.
\end{rem}

\begin{rem}
The above verification does not indicate how I came up with this Lie algebra. Actually, I originally obtained it not using double extensions, but using another device due to Kath and Olbrich \cite{KaO}, allowing to directly pass from some dimension smaller than $n-2$ to dimension $n$. Given a Lie algebra $\mkl$ and an orthogonal $\mkl$-module $\mka$, and given a 2-cycle $\alpha\in Z^2(\mkl,\mka)$ and a 3-chain $\gamma\in C^3(\mkl)$ such that $2d\gamma=\la\alpha\we\alpha\ra$, we can construct a certain regular quadratic Lie algebra $\mkd_{\alpha,\gamma}(\mkl,\mka)$ whose underlying space is $\mkl\oplus\mka\oplus\mkl^*$. 

In the above example, $\mkl$ is a 5-dimensional Lie algebra with basis $(E_3,Y_1,Y_4,Z_3,Y_5)$ (free 3-nilpotent on 2 generators $E_3$, $Y_1$) with nonzero brackets $[E_3,Y_1]=Y_4$, $[E_3,Y_4]=Z_3$, $[Y_1,Y_4]=Y_5$, $\mka$ is a 2-dimensional trivial $\mkl$-module with orthonormal basis $(Z_6,Y_6)$, $\gamma=E_3^*\we Y_4^*\we Y_5^*$, and $\alpha=(E_3^*\we Z_3^*)\otimes Z_6-(Y_1^*\we Y_5^*)\ot Y_6$ (here $d\gamma=0=\langle \alpha\we\alpha\ra$). We refer to \cite{KaO} (or \cite{Kat}) for precise definitions, which we avoided here by rather describing $\g$ as a double extension.
\end{rem}

\section{Solvable Lie algebras}\label{solvable}

\subsection{A solvable Lie algebra with nonzero reduced Koszul map} 

Consider the 9-dimensional regular quadratic Lie algebra defined on the basis $(x,y,z,y',x',$ $u^1,u^{-1},v^1,v^{-1})$ with scalar product defined as
\[\la x,x'\ra=\la y,y'\ra=\la z,z\ra=\la u^1,u^{-1}\ra=\la v^1,v^{-1}\ra=1\]
(other scalar products being zero) and nonzero brackets 
\[[x,y]=z,\;[z,x]=y',\;[y,z]=x',\]
\[[x,u^e]=eu^e,\;[y,v^e]=ev^e,\quad e=\pm 1,\]
\[[u^1,u^{-1}]=x',\;[v^1,v^{-1}]=y';\]
it is convenient to endow it with the Cartan grading in $\mathbf{Z}^2$ for which $u^e$ has degree $(e,0)$, $v^e$ has degree $(0,e)$, and the other five generators have degree $(0,0)$. It is center-by-metabelian, namely its second derived subalgebra is generated by the central elements $x',y'$.

Note that $[h,u^e]=e\la h,x'\ra u^e$ for all $h\in\g_0$ and $e=\pm 1$. 
The Jacobi identity can be checked in total degree $(1,0)$ by
\begin{align*}
\mathrm{jac}(h,h',u^1)= & [h,[h',u^1]]-[h',[h,u^1]]+[u^1,[h,h']]\\
		= & \la h,x'\ra\la h',x'\ra u^1 - \la h',x'\ra\la h,x'\ra u^1+0=0;
\end{align*}
and by symmetry this also holds in every total degree $(e,0)$ or $(0,e)$ for $e=\pm 1$. In total degree zero, it holds when all three elements have degree zero because we recognize the free 3-nilpotent Lie algebra on 2 generators, and otherwise by symmetry it boils down to
\begin{align*}\mathrm{jac}(u^{-1},h,u^1)= & [u^{-1},[h,u^1]]-[u^1,[h,u^{-1}]]+[h,[u^1,u^{-1}]]\\
=& \la h,x'\ra [u^{-1},u^1]+\la h,x'\ra[u^1,u^{-1}]+[h,x']=0.
\end{align*}

We also need to check that the scalar product is invariant. Note that it has degree zero for the Cartan grading, so again we only need to check in total degree zero. When all three elements have degree zero, this is a standard fact about the 5-dimensional regular quadratic Lie algebra $\g_0$. Otherwise, it boils down by symmetry to
\[\la u^{1},[u^{-1},h]\ra=\la u^1,\la h,x'\ra u^{-1}\ra=\la x',h\ra=\la [u^{1},u^{-1}],h\ra.\]

So $\g$ is indeed a regular quadratic Lie algebra. It is 3-solvable (actually center-by-metabelian).

Define the 3-chain
\[c=x\we y\we z-u^1\we u^{-1}\we x-v^{1}\we v^{-1}\we y.\]
Then $c$ is a 3-cycle: indeed we have 
\[\bd_2(x\we y\we z)=x\we x'+y\we y';\]
\[\bd_2(u^{1}\we u^{-1}\we x)=x\we x',\;\bd_2(v^{1}\we v^{-1}\we y)=y\we y'.\]

Again consider the alternating trilinear map $J:(x,y,z)\mapsto\la [x,y],z\ra$. Then
\[J(c)=\la z,z\ra-\la x',x\ra-\la y',y\ra=-1.\]
Thus $\kappa(c)\neq 0$ (and $J$ is a non-exact 3-cocycle). To summarize:

\begin{thm}
The above 9-dimensional regular quadratic Lie algebra $\g$ is center-by-metabelian, and its associated alternating 3-form is non-exact; in particular $\g$ has a nonzero reduced Koszul map.
\end{thm}

\subsection{Smaller-dimensional solvable Lie algebras}

We now check that the above is optimal in dimension. 

\begin{thm}\label{soldi}
Every solvable Lie algebra of dimension $\le 8$ over a field of characteristic zero has a zero reduced Koszul map.
\end{thm}

The proof does not rely on classification (except in the already settled nilpotent case), although it roughly describes the possible structures; in all cases, we produce a grading in $\mathbf{N}$ whose 0-component is abelian and apply Corollary \ref{gradnn}.

First observe that this theorem immediately reduces to the case of an algebraically closed field $K$, as we now assume. Also, by Lemma \ref{somq}, we can suppose that $\g$ is a finite-dimensional regular quadratic Lie algebra.

To prove this theorem, we recall that for any finite-dimensional Lie algebra $\g$ over $K$, a Cartan subalgebra is a nilpotent subalgebra equal to its normalizer; all Cartan subalgebras are conjugate \cite[VII.\S 2]{Bou}. The Cartan grading (relative to a Cartan subalgebra $\mk{h}$) is a grading $\g=\bigoplus\g_\alpha$, where the weights live in $\Hom(\mkh,K)$ (Lie algebra homomorphisms), characterized by the fact that for every $h\in\mk{h}$, the subspace $\g_\alpha$ is contained in the characteristic subspace of $\g$ for the operator $\mathrm{ad}(h)$ and the eigenvalue $\alpha(h)$; we have $\g_0=\mk{h}$.

From now on, $\g$ is endowed with a Cartan grading. An easy observation is that any invariant scalar product has degree 0 (i.e., $\la\g_\alpha,\g_\beta\ra=0$ whenever $\alpha+\beta\neq 0$). In particular, the given scalar product induces a duality between $\g_\alpha$ and $\g_{-\alpha}$ for all $\alpha$. In particular, the codimension of the Cartan subalgebra $\g_0$ is even. (See Astrakhancev \cite{As} for a systematic study of regular quadratic Lie algebras based on a Cartan grading.)

For each nonzero weight $\alpha$ (in the sense that $\g_\alpha\neq 0$), define the {\em coweight} $\alpha^\vee$ as the unique element in $\g_0$ such that $\alpha(h)=\la \alpha^\vee,h\ra$ for all $h\in\g_0$. Note that $\alpha$ vanishes on $[\g_0,\g_0]$, and thus $\alpha^\vee$ is orthogonal to $[\g_0,\g_0]$, which in turn implies that all coweights belong to the center of $\g_0$.

\begin{lem}\label{ifso}
The nilpotent radical of $\g$ is contained in the intersection of all the kernels of weights (each weight being viewed as a linear form on $\g$ vanishing on $\g_\beta$ for every $\beta\neq 0$). In particular, if $\g$ is solvable, then $[\g,\g]\cap\g_0\subset\Ker\,\alpha$ for every weight $\alpha$.
\end{lem}
\begin{proof}
If $\alpha$ is a weight and $x\notin\Ker(\alpha)$, then the ideal $I$ generated by $x$ contains $\g_\alpha$ and since $[x,\g_\alpha]=\g_\alpha$ we deduce that $I$ is not nilpotent. We deduce that the nilpotent radical is contained in $\Ker(\alpha)$ for every~$\alpha$.
\end{proof}

\begin{rem}If $\g$ is solvable, then the nilpotent radical of $\g$ is actually equal to the intersection $\mk{i}$ of all the kernels of weights: to show the other inclusion amounts to showing the latter is nilpotent. Indeed, the quotient of $\g$ by its center $\mk{z}$ can be realized by the adjoint representation as an upper triangular Lie subalgebra of $\mk{gl}_n$ for $n=\dim(\g)$, and the weights are precisely the diagonal elements; thus $\mk{i}/\mk{z}$ is nilpotent, as a strictly upper triangular Lie algebra of matrices, and hence $\mk{i}$ is nilpotent as well.
\end{rem}

\begin{lem}\label{g0ab}
Suppose that $\g$ is solvable. Also suppose that $\g_0$ has codimension $\le 7$ and is abelian. Then $\g$ has a zero reduced Koszul map. 
\end{lem}
\begin{proof}
Let $\mk{i}\subset\g_0$ be the intersection of kernels of all weights.
Let $\mk{v}_0$ be a supplementary subspace of $\mk{i}$ in $\g_0$. 
Denote by $\g_{\neq 0}$ the direct sum of all $\g_\alpha$ for $\alpha\neq 0$. 

Let us first assume that $\g_0$ has codimension at most 5.
We first claim that $[\g_{\neq 0},\g_{\neq 0}]$ is contained in $\g_0$. Indeed, otherwise there are nonzero weights $\alpha,\beta$ such that $\alpha+\beta\neq 0$ and $[\g_\alpha,\g_\beta]\neq 0$. Thus $\alpha+\beta$ is a nonzero weight. If $\alpha\neq\beta$, then $\alpha,\beta,\alpha+\beta$ are distinct and also distinct from their negatives, thus $\g_0$ has codimension at least 6, contradiction. Otherwise $\alpha=\beta$; thus $\alpha$ and $2\alpha$ are weights, thus by the codimension assumption $\g_\alpha$ is at most 1-dimensional, which implies that $[\g_\alpha,\g_\alpha]=0$, contradiction again. (Note that this argument, which does not use solvability, fails for $\mk{sl}_3$ in which $\g_0$ has codimension 6.)

 By the above claim and Lemma \ref{ifso}, we get $[\g_{\neq 0},\g_{\neq 0}]\subset\mk{i}$ (we use the solvability assumption here, invoking Lemma \ref{ifso}).

Then, since $\g_0$ is abelian, $[\mk{v}_0,\mk{v}_0]=0$ and $[\mk{v}_0,\g_{\neq 0}]\subset \g_{\neq 0}$. Thus $\g$ admits a grading in $\{0,1,2\}$, namely $\mk{v}_0\oplus \g_{\neq 0}\oplus\mk{i}$, with abelian 0-component (and central 2-component). It follows from Corollary \ref{gradnn} that $\g$ has a zero reduced Koszul map.

Now suppose that $\g_0$ has codimension 6. If $[\g_{\neq 0},\g_{\neq 0}]\subset\g_0$, the previous argument works. Otherwise, there exist nonzero weights $\alpha,\beta,\gamma$ such that $\alpha+\beta=\gamma$ and $[\g_\alpha,\g_\beta]\neq 0$. We see that the nonzero weights are exactly, counting the possible multiplicity ($\alpha$ and $\beta$ may be equal), $\pm\alpha$, $\pm\beta$, $\pm\gamma$. 

Pick a nonzero element $x_\alpha\in\g_\alpha$ and a non-collinear element $x_\beta\in\g_\beta$; define $x_\gamma=[x_\alpha,x_\beta]$. Consider the dual basis $(x_{-\alpha},x_{-\beta},x_{-\gamma})$ with respect to the scalar product. Then writing $\la [x_\alpha,x_{\beta}],x_{-\gamma}\ra=1$ and using invariance, we obtain $[x_\beta,x_{-\gamma}]=x_{-\alpha}$ and $[x_{-\gamma},x_\alpha]=x_{-\beta}$. We claim that $[x_{-\alpha},x_{-\beta}]=0$: indeed, otherwise it would be a multiple of $x_{-\gamma}$, and this would contradict the nilpotency of $[\g,\g]$. Hence, by invariance, we also have $[x_{-\beta},x_\gamma]=[x_{\gamma},x_{-\alpha}]=0$. Define $\mk{v}_0$ as above; define $\mk{v}_1$ to have basis $(x_\alpha,x_\beta,x_{-\gamma})$; define $\mk{v}_2$ to have basis $(x_{-\alpha},x_{-\beta},x_{\gamma})$ and $\mk{v}_3=\mk{i}$. Then $\g=\bigoplus_{i=0}^3\mk{v}_i$ and $[\mk{v}_i,\mk{v}_j]\subset \mk{v}_{i+j}$ for all $i,j$: this is clear when $ij=0$ or when $\max(i,j)=3$; this has just been checked for $i=j\in\{1,2\}$; when $\{i,j\}=\{1,2\}$ this follows from Lemma \ref{ifso}. Accordingly $\g$ has a grading in $\{0,1,2,3\}$ with abelian 0-component and again by Corollary \ref{gradnn}, $\g$ has a zero reduced Koszul map.

Recalling that $\g_0$ has even codimension, the proof is complete. 
\end{proof}

\begin{proof}[End of proof of Theorem \ref{soldi}]
If $\g_0$ is abelian, Lemma \ref{g0ab} applies. Now we suppose that $\g_0$ is not abelian; hence being quadrable, $\g_0$ has dimension $\ge 5$.
The codimension of $\g_0$ being even, it is equal to 0 or 2.

If $\g_0=\g$, then $\g$ is nilpotent and Theorem \ref{everyni} applies. Otherwise the codimension of $\g_0$ is 2 and thus $\g_0$ has dimension 5 or 6. Denote by $\pm\alpha$ the nonzero weights. By the classification of nilpotent quadrable Lie algebras, $\g_0$ is isomorphic to either $\mk{w}(3)$, $\mk{w}(4)$, or $\mk{w}(3)\times\mk{a}$, where $\mk{a}$ is a 1-dimensional abelian Lie algebra. Since $[\g,\g]$ is nilpotent, we see that $[\g_\alpha,\g_{-\alpha}]$ is contained in $\Ker(\alpha)$; moreover from the Jacobi identity it follows that $[\g_\alpha,\g_{-\alpha}]$ is the line generated by $\alpha^\vee$. In all 3 cases, we deduce that there exists a Lie algebra grading on $\g_0$, denoted $\g_0=\mk{v}_0\oplus\mk{v}_1\oplus \mk{v}_2\oplus\mk{v}_3$, such that $\Ker(\alpha)=\mk{v}_1\oplus \mk{v}_2\oplus\mk{v}_3$ and $\alpha^\vee\in\mk{v}_3$. Thus if we extend this grading to $\g$ by requiring $\g_\alpha$ to have degree 1 and $\g_{-\alpha}$ to have degree 2, we obtain a grading of $\g$ in $\{0,1,2,3\}$ with abelian 0-component, and thus Corollary \ref{gradnn} applies once again.
\end{proof}

\appendix

\section{$H_2$ of current Lie algebras: an example}\label{appcheck}

Here we provide an example of computation of $H_2$ of current Lie algebras, which corroborates the computation by Neeb and Wagemann \cite{NW} but contradicts that of Zusmanovich \cite{Zus}.

Let us consider the following complex 6-dimensional Lie algebra $\mkl$ defined as the semidirect product of $\mk{sl}_2$ with its coadjoint representation. Namely, we consider a basis $(e_1,e_0,e_{-1})$ for $\mk{sl}_2$, with $[e_0,e_{\pm 1}]=\pm e_{\pm 1}$ and $[e_1,e_{-1}]=e_0$. Then denoting by $(E_{-1},E_0,E_{1})$ the dual basis of the dual, a basis for $\mkl$ is given by $(e_1,e_0,e_1,E_1,E_0,E_{-1})$ and its Lie algebra law is described by the nonzero brackets, for $\eps\in\{\pm 1\}$:
\[[e_0,e_{\epsilon}]=\epsilon e_{\epsilon},\;\; [e_1,e_{-1}]=e_0,\;\; [e_0,E_\epsilon]=\epsilon E_\epsilon,\;\;[e_{\epsilon},E_0]=-\epsilon E_{\epsilon},\;\; [e_\epsilon,E_{-\epsilon}]=E_0.\]

\begin{lem}\label{H20ver}
$H_2(\mkl)=0$.
\end{lem}
\begin{proof}
This can be checked by a lengthy computation, but here is a short argument, only using that $\mkl=\mks\ltimes\mkv$ where $\mks$ is semisimple and $\mkv$ is a nontrivial irreducible representation of $\mks$ with a nonzero symmetric bilinear form. We have to show that if $\tilde{\mkl}$ is a central extension of $\mkl$ by a line $\mk{z}$, then it is a direct product. Indeed, let $\tilde{\mkv}$ be the radical of $\tilde{\mkl}$. Then $\tilde{\mkv}$ is a central extension of $\mkv$ by $\mk{z}$, and the Lie bracket on $\tilde{\mkv}$ provides a $\mks$-invariant alternating form $\mkv\times\mkv\to\mk{z}$. Since $\mk{v}$ is an irreducible $\mks$-module (and the ground field is $\mathbf{C}$), the space of $\mks$-invariant forms is at most one-dimensional. Hence $\tilde{\mkv}$ is abelian. As an $\mks$-module, it admits $\mkv$ as a quotient. By semisimplicity, it follows that, as an $\mks$-module, $\tilde{\mkv}=\mkv\oplus\mk{z}$. Thus the central extension is a direct product. 
\end{proof}

As any semidirect product of a Lie algebra by its coadjoint representation, there is a natural invariant scalar product $B=\la\cdot,\cdot\ra$ on $\mkl$, given by $\la e_i,E_{-i}\ra=1$ for $i\in\{1,0,-1\}$ and other scalar products zero.

In the following proposition, the ground ring is $\mathbf{C}$, and $\mathbf{C}[t]$ is the polynomial ring on the indeterminate $t$.

\begin{prop}\label{proph}
We have $H_2(\mkl\otimes\mathbf{C}[t])\neq\{0\}$; namely $te_1\we E_{-1}-e_1\we tE_{-1}$ represents a nonzero element in homology. 
\end{prop}

This result can be proved using the general results of \cite{NW}. On the other hand, it is in contradiction with the main result of \cite{Zus}, see below. Therefore we will provide an explicit, computational, proof. Also let us mention that although it is not in contradiction with \cite{Had}, it does not follow from it. (See below.)

\begin{proof}
We are going to check the result by replacing $\mathbf{C}[t]$ with an arbitrary commutative $\mathbf{C}$-algebra $A$ with a nonzero $\mathbf{C}$-linear self-derivation denoted by $x\mapsto x'$.

The following construction is inspired from Abels' computation in \cite[\S 5.7.4]{Ab}. Define a $\mathbf{C}$-linear map $A\ot A\to A$ by $\rho(\lambda\ot\mu)=\lambda'\mu$. A straightforward computation shows that it satisfies:
\[\rho(\lambda\mu\ot\alpha)=\rho(\lambda\ot\mu\alpha+\mu\ot\alpha\lambda),\quad\forall\lambda,\mu,\alpha\in A.\]

Given an element of $\g=A\ot\mkl$, we write it $ax$ instead of $a\ot x$. Note that the linear decomposition $\mkl=\mks\oplus\mkv$ (which actually defines a Lie algebra grading in $\{0,1\}$) induces a $\mathbf{C}$-linear isomorphism
\[\Lambda^2(\g)=\Lambda^2(A\ot\mks)\oplus \big((A\ot\mks)\otimes (A\ot\mkv)\big)\oplus \Lambda^2(A\ot\mkv).\]
Define a $\mathbf{C}$-linear map $c:\Lambda^2(A\ot\mkl)\to A$ by
\[c(\lambda e_i\we \mu E_{-i})=\rho(\lambda\ot\mu),\quad i=1,0,-1,\;\lambda,\mu\in A\]
and $c$ is zero on other basis elements, that is, $c$ vanishes on both $\Lambda^2(A\ot\mks)$ and $\Lambda^2(A\ot\mkv)$ and on elements of the form $\lambda e_i\we \mu E_j$ for $i\neq -j$.

Now let $t\in A$ be an element such that $t'\neq 0$. For instance, in $\mathbf{C}[t]$ with the standard derivation, $t$ itself works. Then 
\[c(te_1\we E_{-1}-e_1\we tE_{-1})=\rho(t\ot 1-1\ot t)=t'1-1't=t'\neq 0.\]

Clearly $te_1\we E_{-1}-e_1\we tE_{-1}$ is a 2-cycle (since the bracket is $A$-bilinear). So to conclude that $H_2(\g)\neq 0$, it is enough to check that $c$ vanishes on 2-boundaries.

To check this, we use a convenient grading of $\mkl$ in $\mathbf{Z}^2$: $e_i$ has degree $(i,0)$ and $E_i$ has degree $(i,1)$. This grading induces a grading on the exterior algebra and on the homology, and also passes to the current algebra (defining $\g_\omega=A\ot\mkl_\omega$ for $\omega\in\mathbf{Z}^2$) and its own exterior algebra. Note that $c$ vanishes in every degree $\neq (0,1)$. Therefore all we need to check is that $c$ vanishes on boundaries of degree $(0,1)$. 
A basis of $\mkl_{(0,1)}$ is $(e_1\we e_0\we E_{-1},e_{-1}\we e_0\we E_1,e_1\we e_{-1}\we E_0)$. Thus $\g_{(0,1)}$ is linearly generated by $A$-multiples of the same elements. We compute ($\lambda,\mu,\alpha\in A$):
\begin{align*}
\lambda e_1\we \mu e_0\we \alpha E_{-1}   \stackrel{d}\mapsto & -\lambda e_{1}\we \mu\alpha E_{-1}-
\mu e_{0}\we \alpha\lambda E_{0}+\lambda\mu e_{1}\we \alpha E_{-1} \\
& \stackrel{c}\mapsto \rho(-\lambda\ot\mu\alpha-\mu\ot\alpha\lambda+\lambda\mu\ot\alpha)=0; \\
\end{align*}
\begin{align*}\lambda e_{-1}\we \mu e_{0}\we \alpha E_{1}  \stackrel{d}\mapsto & \dots\stackrel{c}\mapsto 0\quad (\textnormal{similarly});\\
\lambda e_{1}\we \mu e_{-1}\we \alpha E_{0}  \stackrel{d}\mapsto & \;\lambda e_{1}\we \mu\alpha E_{-1}+\mu e_{-1}\we \alpha\lambda E_{1}-\lambda\mu e_{0}\we \alpha E_{0}\\
  &\stackrel{c}\mapsto \rho(\lambda\ot\mu\alpha+\mu\ot\alpha\lambda-\lambda\mu\ot\alpha)=0,
\end{align*}
hence $c$ indeed vanishes on 2-boundaries. Accordingly, $te_1\we E_{-1}-e_1\we tE_{-1}$ defines a nonzero element of $H_2(\g)$.\end{proof}

Now let us compare with the descriptions in \cite{Had,Zus,NW}. All these papers, given a complex Lie algebra $\mk{h}$ (all these papers work in a greater generality on the ground field, but let us stick to $\mathbf{C}$), address the problem of describing $H_2(A\ot\mk{h})$ when $A$ is an arbitrary associative unital $\mathbf{C}$-algebra.

The results in all these papers significantly simplify when we assume $H_1(\mk{h})=H_2(\mk{h})=0$, as we now assume. 

The result in \cite{Had} provides an exact sequence
\begin{equation}\label{hadseq}0\to D(A,\mk{h})\to H_2(A\ot\mk{h})\to \HC_1(A)\ot\Kill(\mk{h})\to 0;\end{equation}
the result in \cite{Zus} asserts an isomorphism
\begin{equation}\label{zusis}
H_2(A\ot\mk{h})\simeq \HC_1(A)\ot\Kill(\mk{h});
\end{equation}
finally the result in \cite{NW} describes $H_2(A\ot\mk{h})$ as the cokernel of a certain operator $A\times\mk{h}^{\ot 3}\to\Lambda^2A\ot S^2\mk{h}\oplus A\ot Z_2(\mk{h})$. Above, $\HC_1$ is the first cyclic homology of $A$, which can be described as the quotient of $A\we A$ by the subspace generated by elements of the form $\lambda\we\mu\alpha+\mu\we\alpha\lambda+\alpha\we\lambda\mu$. 

Note that the blatant difference between (\ref{hadseq}) and (\ref{zusis}) is concealed by the existence of several additional terms related to $H_1(\mk{h})$ and $H_2(\mk{h})$. Now let us compare Proposition \ref{proph} with the above description: when $A=\mathbf{C}[t]$, we actually have $\HC_1(A)=0$, which can be checked by hand, see also \cite[Example 3.1.7]{Lod}. In (\ref{hadseq}), the right-hand term vanishes, and this is in keeping with the fact that the nontrivial homology class $tx\we X-x\we tX$ belongs to $D(A,\mk{h})$, on the other hand \cite{Had} only provides generators of $D(A,\mk{h})$ and in particular cannot predict its non-vanishing. 

The assertion (\ref{zusis}) is contradictory, since it erroneously implies that $H_2(\mathbf{C}[t]\ot\mk{l})=0$ for the above 6-dimensional Lie algebra $\mk{l}$. Actually, for a Lie algebra $\mk{h}$ with $H_1(\mk{h})=H_2(\mk{h})=0$, the vanishing of $H_2(\mathbf{C}[t]\ot\mk{l})$ only holds (as a consequence of results of \cite{NW}) when the reduced Koszul map is surjective (for instance, when $\mk{l}$ is semisimple). We explain this in detail in Appendix \ref{nwd}, coming back to this example in Example \ref{scoadj}.


\newpage

\section{The Neeb-Wagemann description of 2-homology, revisited}\label{nwd}

The purpose of this appendix is to rewrite the results of \cite{NW} in a more convenient way, so as to decompose canonically the second homology of current algebras into explicit spaces.

The setting here is the following: $\mkl$ is a Lie algebra over a fixed field $K$ of characteristic $\neq 2$; whenever we use language pertaining to linear algebra (linearity, tensor products, homology etc.), the ground ring is $K$. Let us fix an associative, unital, commutative $K$-algebra $A$.

Define two maps $T,T_0:A^{\ot 3}\to \Lambda^2(A)$ by
\[T(a\ot b\ot c)=ab\we c+bc\we a+ca\we b;\quad T_0(a\ot b\ot c)=T(a\ot b\ot c)-abc\we 1.\]
Define the first cyclic homology $\HC_1(A)=\mathrm{coker}(T)$ and the first Hochschild homology $\HH_1(A)=\mathrm{coker}(T_0)$ (see \cite[Section 2]{NW} for a comparison with more standard definitions; these ones assume that $A$ is commutative). A simple verification shows that the image of $T_0$ is contained in the image of $T$, and therefore we have successive quotients \[\Lambda^2(A)\twoheadrightarrow \HH_1(A)\twoheadrightarrow \HC_1(A).\]

Let $I_A$ be the kernel of the multiplication map $S^2A\to A$.
A simple but very useful observation of \cite{NW} is the following:

\begin{prop}\label{candeco}
There is a canonical linear isomorphism
\begin{align*}
\Lambda^2(A\ot\mkl) &\to& \Lambda^2A\ot S^2\mkl &&\oplus& A\ot\Lambda^2\mkl &\oplus & I_A\ot \Lambda^2\mkl\\
ax\we by &\mapsto  &(a\we b)\ot (x\cc y)& &+& ab\ot (x\we y) &+& (a\cc b-ab\cc 1)\ot (x\we y);\end{align*}
it restricts to a linear isomorphism
\[Z_2(A\ot\mkl) \;\;\to\;\; \Lambda^2A\ot S^2\mkl \;\;\oplus\;\; A\ot Z_2(\mkl) \;\;\oplus \;\; I_A\ot \Lambda^2\mkl.\]
\end{prop}

We henceforth use the above isomorphism as an identification.

If we write the above decomposition of $\Lambda^2(A\ot\mkl)$ as $V_1\oplus V_2\oplus V_3$, in \cite{NW}, the authors then study 2-boundaries and describe them as the sum of 4 subspaces, the first two of which being contained in the first summand $V_1$, the third (fourth in \cite{NW}) is contained in $V_4$ and the fourth is contained in $V_1\oplus V_2$, showing that in general the 2-homology does not split as a product with respect to this decomposition. Still, it is convenient to restate things by first modding out by the three ``homogeneous terms" and then study the last one. Let us first state the theorem in \cite{NW}

\begin{thm}[Neeb-Wagemann, {\cite[Theorem 3.4]{NW}}]\label{NWt}
Define the linear map
\[f: A\ot \mkl^{\ot 3}\to \Lambda^2A\ot S^2\mkl \;\;\oplus\;\; A\ot B_2(\mkl) \;\;\subset\;\;\Lambda^2(A\ot\mkl)\] 
by $f=(f_1,f_2)$, where
\[f_1(a\ot x\ot y\ot z)=(a\we 1)\ot ([x,y]\cc z),\;\; f_2(a\ot x\ot y\ot z)=a\ot \bd_3(x\we y\we z).\]
Then 
\[B_2(A\ot\mkl)=W_1+W'_1+W_3+W_{12},\]
where
\[W_1=\Lambda^2A\ot\mkl.S^2\mkl\subset V_1,\quad W'_1=T_0(A^{\ot 3})\ot (\mkl\cc[\mkl,\mkl])\subset V_1,\]
\[W_3=I_A\ot (\mkl\we [\mkl,\mkl])\subset V_3,\quad W_{12}=\mathrm{Im}(f)\subset V_1\oplus V_2.\]
\end{thm}

Here $\mkl.S^2\mkl$ is the kernel of the natural projection $S^2\mkl\to\Kill(\mkl)$.

Neeb and Wagemann then give applications of this theorem to the existence of ``coupled cocycles", in relation to the Koszul map; see Corollary \ref{coucoc} for a statement in terms of the Koszul map.

The quotient $V_3/W_3$ is $I_A\ot\Lambda^2(H_1(\mkl))$. Thus the 2-homology can be described as 
\[H_2(A\ot\mkl)=\Big((V_1\oplus V_2)/(W_1+W'_1+\mathrm{Im}(f))\Big)\oplus \Big(I_A\ot\Lambda^2(H_1(\mkl))\Big)\]

Define
\[E_1=V_1/W_1=\Lambda^2(A)\ot\Kill(\mkl),\quad E_2=V_2=A\otimes Z_2(\mkl).\]
Note that $f_1$ composed with the projection onto $E_1$ factors through $A\otimes\Lambda^3\mkl$, there the component $\Lambda^3\mkl\to\Kill(\mkl)$ is nothing else than the Koszul map $\eta=\eta_\mkl$ of $\mkl$; note that $f_2$ already factors through $A\ot\Lambda^3\mkl$.

In addition, define $M=(A\oplus A^{\ot 3})\ot\Lambda^3\mkl$ and a $K$-linear map $u:M\to E_1\oplus E_2$ by $u=(u_1,u_2)=(v_1\oplus v'_1,u_2\oplus 0)$ with
\begin{align*}v_1\big(a,(b\ot c\ot d))\otimes (x\we y\we z)\big)=& (a\we 1)\ot \eta(x\we y\we z);\\
v'_1\big(a,(b\ot c\ot d))\otimes (x\we y\we z)\big)=& T_0(b\ot c\ot d)\ot \eta(x\we y\we z);\\
u_2\big(a,(b\ot c\ot d))\otimes (x\we y\we z)\big)=& a\ot \bd(x\we y\we z).
\end{align*}

\begin{thm}[Neeb-Wagemann, first restatement]\label{nwfr}
There is a canonical linear isomorphism
\[H_2(A\ot\mkl)\simeq \big((E_1\oplus E_2)/\mathrm{Im}(u)\big)\oplus \big(I_A\otimes \Lambda^2H_1(\mkl)\big).\]
\end{thm}

Note that since the kernels of $(v_1,u_2)$ and $(v'_1,0)$ generate $M$,
clearly $\mathrm{Im}(u)=\mathrm{Im}(v_1,u_2)+\mathrm{Im}(v'_1)$; the first term being the projection of $\mathrm{Im}(f)$ and the second being the projection of $W'_1$. Therefore the above theorem indeed follows from Theorem \ref{NWt}.

Now to make the theorem more explicit, we have to describe $\coker(u)$.
The following general linear algebra lemma is immediate, but is useful to have in mind to follow the argument.
\begin{lem}\label{linalg}
Let $u=(u_1,u_2):M\to E_1\oplus E_2$ be a linear map between vector spaces. Then there is a canonical exact sequence:
\[0\to M/(\Ker(u_1)+\Ker(u_2))\to \coker(u)\to \coker(u_1)\oplus\coker(u_2)\to 0.\]
Moreover, there are canonical isomorphisms 
\[M/(\Ker(u_1)+\Ker(u_2))\simeq \mathrm{Im}(u_1)/u_1(\Ker(u_2))\simeq \mathrm{Im}(u_2)/u_2(\Ker(u_1)).\]

Besides, the image of $u$ splits according to the direct decomposition if and only if the kernel vanishes, i.e.\ $M=\Ker(u_1)+\Ker(u_2)$.
\end{lem}
\begin{proof}
(Note that all this holds for modules over an arbitrary ring, and probably can be adapted to make sense in an arbitrary abelian category.)

Consider the map $(u_1,0)$ and view it after composition as a map $u'$ from $M$ to $\coker(u)$. Then $u'$ vanishes on both $\Ker(u_1)$ and $\Ker(u_2)$ and hence factors through a map $u'':M/(\Ker(u_1)+\Ker(u_2))\to\coker(u)$. Obviously composing with the projection onto $\coker(u_1)\oplus\coker(u_2)$ yields zero. Let us check exactness: suppose that some element $y=(y_1,y_2)\in E_1\oplus E_2$ has a trivial image in $\coker(u_1)\oplus\coker(u_2)$. This means that we can write $(y_1,y_2)=(u_1(x_1),u_2(x_2))$ for some $x_1,x_2\in M$. Then $y-u(x_2)=(u_1(x_1-x_2),0)$; thus the image of $y$ in $\coker(u)$ belongs to the image of $u''$.

For the second statement, consider the map $u_1$ and view it, after composition, as a map $M\to \mathrm{Im}(u_1)/u_1(\Ker(u_2))$. Then it obviously vanishes on both $\Ker(u_1)$ and $\Ker(u_2)$ and thus induces a map $M/(\Ker(u_1)+\Ker(u_2))\stackrel{v}\to \mathrm{Im}(u_1)/u_1(\Ker(u_2))$; this map is surjective by definition; if $x\in M$ represents an element in the kernel, it means that $u_1(x)=u_1(y)$ for some $y\in\Ker(u_2)$ and thus $x=(x-y)+y\in\Ker(u_1)+\Ker(u_1)$; thus $v$ is an isomorphism.

As for the third assertion, the image splits by definition if and only if $\mathrm{Im}(u)=(\mathrm{Im}(u)\cap E_1)+(\mathrm{Im}(u)\cap E_2)$. Since $u^{-1}(\mathrm{Im}(u)\cap E_i)=\Ker(u_i)$, and since the inverse image map $u^{-1}$ between power sets is injective on $\mathrm{Im}(u)$, this is equivalent to the requirement $M=\Ker(u_1)+\Ker(u_2)$.
\end{proof}

Therefore, turning back to the setting of Theorem \ref{nwfr}, we have to describe all three subspaces given by Lemma \ref{linalg}: the cokernel of $u_1=v_1\oplus v'_1$ and $u_2$, and the quotient $M/(\Ker(u_1)+\Ker(u_2))$.
The easiest is $\coker(u_2)$, which is just $A\ot H_2(\mkl)$.
The cokernel of $u_1$ can be described by either of the following natural exact sequences: recall from \S\ref{fiki} that $\Kill^{(3)}(\mkl)$ is defined as the image of $\mkl\ot [\mkl,\mkl]$ in $\Kill(\mkl)$:

\[0\to \HC_1(A)\otimes\Kill^{(3)}(\mkl)\to\mathrm{coker}(u_1)\to \Lambda^2A\otimes S^2(H_1(\mkl))\to 0;\]
\[0\to T(A^{\otimes 3})\otimes S^2(H_1(\mkl))\to\mathrm{coker}(u_1)\to \HC_1(A)\otimes\Kill(\mkl)\to 0.\]

Actually $\coker(u_1)$ is an iterated extension of the three modules $T(A^{\otimes 3})\ot S^2(H_1(\mkl))$, $\HC_1(A)\ot S^2(H_1(\mkl))$, and $\HC_1(A)\ot\Kill^{(3)}(\mkl)$, refining both of these exact sequences. Note that if $H_1(\mkl)=0$ then only one term remains, namely the most interesting one: $\HC_1(A)\ot\Kill^{(3)}(\mkl)$.

Finally, we have to describe the kernel $M/(\Ker(u_1)+\Ker(u_2))$. Lemma \ref{linalg} provides 3 descriptions of this space.
The first description is just by using the definition. We have, in the decomposition $M=A\ot\Lambda^3\mkl\oplus A^{\ot 3}\ot\mkl$
\[\Ker(u_2)=A\ot Z_3(\mkl)\oplus A^{\ot 3}\ot\Lambda^3\mkl,\]
thus
\[M/\Ker(u_2)\simeq A\ot B_2(\mkl)\oplus \{0\}.\]

To describe the kernel of $u_1$, we write $w(a\oplus(b\ot c\ot d))=(a\we 1)\oplus T_0(b\ot c\ot d)$, so that $u_1=w\ot\eta$. Then, we have
\[\Ker(u_1)=\Ker(w)\otimes\Lambda^3\mkl+(A\oplus A^{\ot 3})\ot\Ker(\eta);\]
thus 
\[\Ker(u_1)+\Ker(u_2)=A\ot (Z_3(\mkl)+\Ker\,\eta)+(A^{\ot 3}+\Ker(w))\ot\Lambda^3\mkl\]
To describe $A^{\ot 3}+\Ker(w)\subset A\oplus A^{\ot 3}$, it is enough to compute its projection on $A$, which is precisely the kernel $A_0$ of the map $a\mapsto [a\we 1]$ from $A$ to $\HH_1(A)$. Thus 
\[\Ker(u_1)+\Ker(u_2)=A\ot (Z_3(\mkl)+\Ker\,\eta)+(A_0\oplus A^{\ot 3})\ot\Lambda^3\mkl;\]
and hence we have 
\[M/(\Ker(u_1)+\Ker(u_2))= \Big(A/A_0 \ot \big(\Lambda^3\mkl/(Z_3(\mkl)+\Ker\,\eta)\big)\Big)\oplus\{0\}.\] 
 
In particular, using the isomorphisms of Lemma \ref{linalg}, we have natural isomorphisms
\[M/(\Ker(u_1)+\Ker(u_2))\simeq A/A_0 \ot B_2(\mkl)/\bd_3(\Ker\,\eta)\simeq A/A_0\ot \Kill^{(3)}(\mkl)/\mathrm{Im}(\bar{\eta}).\] 

Thus we have the following corollary
\begin{cor}\label{corh2}
The space $H_2(A\ot\mkl)$ is an iterated extension of the following spaces:
\begin{itemize}
\item $\HC_1(A)\ot\Kill^{(3)}(\mkl)$;
\item $(A/A_0)\ot\Kill^{(3)}(\mkl)/\mathrm{Im}(\bar{\eta})$; 
\item $\Lambda^2 A\ot S^2(H_1(\mkl))$;
\item $I_A\ot \Lambda^2(H_1(\mkl))$;
\item $A\ot H_2(\mkl)$.
\end{itemize}
The natural projection $H_2(A\ot\mkl)\to A\ot H_2(\mkl)$ is surjective and its kernel is an iterated extension of the above terms except the last one.
\end{cor}

Note that the surjectivity statement is immediate, although it may fail in higher homology. The above statement significantly simplifies when $H_1$ vanishes, but since this is a quite strong restriction, let us fit it into the graded setting: we assume that the Lie algebra $\mkl=\bigoplus_{\beta\in \mathcal{B}}\mkl_\beta$ is graded in an abelian group $\mathcal{B}$. (All what follows encompasses the non-graded setting just by picking a grading concentrated in degree zero: $\mkl_0=\mkl$.) This induces a grading of $A\ot\mkl=\bigoplus_{\beta\in \mathcal{B}}A\ot\mkl_{\beta}$, and on all exterior algebras and homology spaces; the natural maps (boundary, Koszul map\dots) preserving the grading. Note that the condition $H_1(\mkl)_\beta=\{0\}$ means that $\mkl_\beta\subset [\mkl,\mkl]$.

\begin{cor}\label{h10}
Fix a weight $\beta\in \mathcal{B}$ and assume that $S^2(H_1(\mkl))_\beta=\{0\}$ (i.e., $H_1(\mkl)$ has no opposite weights). Then  
the space $H_2(A\ot\mkl)_\beta$ is an iterated extension of the following spaces:
\begin{itemize}
\item $\HC_1(A)\ot\Kill(\mkl)_\beta$;
\item $(A/A_0)\ot\big(\Kill(\mkl)/\mathrm{Im}(\bar{\eta}_\mkl)\big)_\beta$; 
\item $A\ot H_2(\mkl)_\beta$.
\end{itemize}
The natural projection $H_2(A\ot\mkl)_\beta\to A\ot H_2(\mkl)_\beta$ is surjective and its kernel $\mathcal{K}(A,\mkl)_\beta$ in degree $\beta$ lies in each of the following two exact sequences:
\[0\to (A/A_0)\ot\coker(\bar{\eta}_\mkl)_\beta \to \mathcal{K}(A,\mkl)_\beta\to \HC_1(A)\ot\Kill(\mkl)_\beta\to 0,\]
\[0\to \HC_1(A)\ot\mathrm{Im}(\bar{\eta}_\mkl)_\beta \to \mathcal{K}(A,\mkl)_\beta\to \HH_1(A)\ot\coker(\bar{\eta}_\mkl)_\beta\to 0,\]
\end{cor}

\begin{rem}
By the Pirashvili exact sequence (\cite{Pir}, see also the introduction), the cokernel $\Kill(\mkl)/\mathrm{Im}(\bar{\eta}_\mkl)_\beta$ can be identified with the kernel of the projection $H_1(\mkl,\mkl)_\beta\to H_2(\mkl)_\beta$. In particular, when $H_2(\mkl)_\beta=\{0\}$, the surjectivity of $\bar{\eta}_\mkl$ in degree $\beta$ is equivalent to the vanishing of $H_1(\mkl,\mkl)_\beta$.
\end{rem}

\begin{rem}\label{aa0hh1}
We have $A/A_0\neq\{0\}$ if and only if $\HH_1(A)\neq\{0\}$. Indeed, $A/A_0\subset\HH_1(A)$ so one direction is obvious. Conversely, if $\HH_1(A)\neq\{0\}$, then there exists a nonzero derivation $d$ on $A$ \cite[Proposition 1.1.10]{Lod}; this implies that, for $x\notin\Ker(d)$, the element $x\notin A_0$ (use that the map $a\we b\mapsto ad(b)-bd(a)$ vanishes on $T_0(A^{\ot 3})$ and maps $x\we 1$ to $-d(x)\neq 0$).
\end{rem}

\begin{cor}\label{h2van}
Suppose that the $K$-algebra $A$ is not reduced to $\{0\}$ or $K$. Fix a weight $\beta\in \mathcal{B}$. Then $H_2(A\ot\mkl)_\beta=\{0\}$ if and only if $S^2(H_1(\mkl))_\beta=H_2(\mkl)_\beta=\{0\}$ and
	\begin{itemize}
	\item $\HH_1(A)=\{0\}$ or
	\item $\Kill(\mkl)_\beta=\{0\}$ or
	\item $\bar{\eta}_\mkl$ is surjective in degree $\beta$ and $\HC_1(A)=\{0\}$.
	\end{itemize}
Also, the kernel of $\Ker(H_2(A\ot\mkl)_\beta\to A\ot H_2(\mkl)_\beta)$ is trivial exactly under the same condition, except that we drop the condition of vanishing of $H_2(\mkl)_\beta$.
\end{cor}
\begin{proof}
If these spaces all vanish, observe that $\Lambda^2(H_1(\mkl))_\beta$ vanishes as well, and hence all subfactors in Corollary \ref{h10} vanish and and therefore $H_2(A\ot\mkl)=\{0\}$, and similarly if all but $H_2$ vanish, then the kernel vanishes.

Conversely, suppose that one of the conditions fails:
\begin{itemize}
\item Suppose $H_2(\mkl)_\beta\neq\{0\}$. Since $H_2(A\ot\mkl)_\beta\to A\ot H_2(\mkl)_\beta$ is surjective, it follows that $H_2(A\ot\mkl)_\beta\neq\{0\}$. 
\item Suppose $S^2(H_1(\mkl))_\beta\neq\{0\}$. By the assumption on $A$, we have $\Lambda^2A\neq\{0\}$, and hence the third term in Corollary \ref{corh2} is nonzero.
\item Suppose that all three additional conditions fail and $S^2(H_1(\mkl))_\beta=\{0\}$. We have two cases:
\begin{itemize}
\item $\HH_1(A)\neq\{0\}$, $\Kill(\mkl)_\beta\neq\{0\}$, and $\bar{\eta}_\mkl$ is not surjective in degree $\beta$. Then by Remark \ref{aa0hh1}, $A/A_0$ is nonzero. Hence the second subfactor in Corollary \ref{corh2} is nonzero.
\item $\HC_1(A)\neq\{0\}$, $\Kill(\mkl)_\beta\neq\{0\}$. Then the first subfactor in Corollary \ref{corh2} is nonzero. 
\end{itemize}
\end{itemize}
\end{proof}

\begin{rem}
Under the assumptions of Corollary \ref{h2van}, we have $I_A\neq\{0\}$, which implies that the fourth term in Corollary \ref{corh2} is also nonzero in case $\Lambda^2(H_1(\mkl))_\beta\neq\{0\}$. Indeed, pick $x\in A\smallsetminus K 1_A$, which exists by assumption. If the $K$-algebra $K'$ generated by $x$ is finite-dimensional as a vector space over $K$, then the multiplication map $S^2(K')\to K'$ cannot be injective by a dimension argument. Otherwise, $x$ is transcendent and then $x^2\cc x^2-x\cc x^3$ is a nonzero element in $I_A$.
\end{rem}

Also, using the last statement of Lemma \ref{linalg}, we have:

\begin{cor}\label{coucoc}
The space of 2-boundaries $B_2(A\ot\mkl)_\beta$ splits according to the decomposition of Proposition \ref{candeco} if and only if $\mathrm{Im}(\bar{\eta}_\mkl)_\beta=\Kill^{(3)}(\mkl)_\beta$. 
\end{cor}

The above splitting property is called ``to have no nonzero coupled 2-cocycles" in \cite{NW}. (Corollary \ref{coucoc} is essentially equivalent to Theorem 4.7 in \cite{NW}.)

Notwithstanding the semisimple case, the condition that $\bar{\eta}$ is zero is very frequent (especially at a fixed weight of a given grading since, for $K$ of characteristic zero, it holds notably when $\beta$ is not a torsion element in $\mathcal{B}$, by Theorem \ref{vanish}). Therefore let us make the corollaries explicit in this case:

We begin by Corollary \ref{corh2}: then we can glue the first two terms.

\begin{cor}\label{iterat}
Fix a weight $\beta\in \mathcal{B}$ and suppose that $\bar{\eta}$ is zero in degree $\beta$. Then the kernel $\mathcal{K}(A,\mkl)_\beta$ of the surjective map $H_2(A\ot\mkl)_\beta\to A\ot H_2(\mkl)_\beta$ lies in an exact sequence
\[0\to T_0(A^{\ot 3})\ot S^2(H_1(\mkl))_\beta\to \mathcal{K}(A,\mkl)_\beta\qquad\qquad\qquad\qquad\qquad\qquad\qquad\]
\[\qquad\qquad\to \big(\HH_1(A)\ot\Kill(\mkl)_\beta\big)\oplus \big(I_A\ot \Lambda^2(H_1(\mkl))_\beta\big)\to 0;\]
\end{cor}
 
\begin{cor}\label{h10eta0}
Fix a weight $\beta\in \mathcal{B}$ and assume that $S^2(H_1(\mkl))_\beta=\{0\}$ and that $\bar{\eta}$ is zero in degree $\beta$. Then  
the space $H_2(A\ot\mkl)_\beta$ lies in an exact sequence
\[0\to \HH_1(A)\ot\Kill(\mkl)_\beta \to H_2(A\ot\mkl)_\beta\to A\ot H_2(\mkl)_\beta\to 0.\]
\end{cor}

Let us give a corollary of Theorem \ref{nwfr} in an even more particular case (which in a sense is the less interesting one, but ought to be written because it is often applicable), namely when $\Kill^{(3)}(\mkl)_\beta=\{0\}$. 

\begin{cor}\label{killess0}
Suppose that $\Kill^{(3)}(\mkl)_\beta=\{0\}$. Then we have an isomorphism
\[H_2(A\ot\mkl)_\beta\simeq\big(\Lambda^2(A)\ot S^2(H_1(\mkl))_\beta\big)\oplus \big(A\ot H_2(\mkl)_\beta\big)\oplus \big(I_A\ot \Lambda^2(H_1(\mkl))_\beta\big).\]
\end{cor}
\begin{proof}
The assumption means that the Koszul map (not only the reduced one) vanishes in degree $\beta$, and also means that $\Kill(\mkl)_\beta=S^2(H_1(\mkl))_\beta$. The result immediately follows.
\end{proof}

\begin{ex}A simple example is the 2-dimensional non-abelian Lie algebra $\mkl$. This Lie algebra is convenient to be seen with a grading in $\{0,1\}$. Then both its 1-homology and its Killing module are concentrated in degree zero; its second and third homology vanish. By Corollary \ref{h2van}, it follows that $H_2(A\ot\mkl)_n=0$ for $n=1,2$ regardless of $A$. In degree zero, we just obtain the same as in the case of an abelian Lie algebra, namely $\Lambda^2(A\ot\mkl_0)\simeq\Lambda^2(A)$. Hence
\[H_2(A\ot\mkl)=H_2(A\ot\mkl)_0\simeq\Lambda^2(A),\]
where an isomorphism $\Lambda^2(A)\to H_2(A\ot\mkl)$ is given by the map $\lambda\we\mu\mapsto \lambda x\we \mu x$, where $x$ is a fixed nonzero element of $\mkl_0$.
\end{ex}

\begin{ex}
Let $\mkl$ be a Carnot-graded Lie algebra. Then $S^2(H_1(\mkl))$ is concentrated in degree 2. Thus for $n\ge 3$, by Corollary \ref{iterat} (and Corollary \ref{nstar}), we have an exact sequence
\[0\to \HH_1(A)\ot\Kill(\mkl)_n\to H_2(A\ot\mkl)_n\to A\ot H_2(\mkl)_n\to 0,\]
the left hand homomorphism being induced by $(\lambda\we\mu)\ot (x\cc y)\mapsto \lambda x\we \mu y-\mu x\we \lambda y$.
(The vanishing of $H_2(\mkl)_n$ for all $n\ge 3$ characterizes when $\mkl$ is quadratically presented.) In degree 2, Corollary \ref{killess0} provides an isomorphism
\[H_2(A\ot\mkl)_2\simeq\big(\Lambda^2(A)\ot S^2(\mkl_1)\big)\oplus \big(A\ot H_2(\mkl)_2\big)\oplus\big( I_A\ot\Lambda^2(\mkl_1)\big).\]
\end{ex}

\begin{ex}\label{scoadj}
Let $\mk{s}$ be a simple complex Lie algebra and $\mkl=\mk{s}\ltimes\mk{s}^*$ the coadjoint semidirect product. Then $H_1(\mkl)=H_2(\mkl)=0$ (see the proof of Lemma \ref{H20ver}). Consider the grading of $\mkl$ in $\{0,1\}$, with $\mk{s}$ of degree 0 and the abelian ideal $\mk{s}^*$ of degree 1. Then $\Kill(\mkl)$ is 2-dimensional, namely 1-dimensional in both degree 0 and degree 1 (and zero in degree 2, by a straightforward argument), and $\bar{\eta}$ is surjective in degree 0, and is zero in degree 1 (the former by Koszul's theorem, and the latter by an easy instance of Theorem \ref{H20ver}). Therefore, by Corollary \ref{h10}, we have 
\[H_2(A\ot\mkl)_0\simeq\HC_1(A),\]
and by Corollary \ref{h10eta0}, we obtain
\[H_2(A\ot\mkl)_1\simeq\HH_1(A),\quad H_2(A\ot\mkl)_2=0.\]
\end{ex}


\end{document}